\documentclass [11pt]{amsart}

\usepackage{graphicx}
\usepackage{amsmath}
\usepackage{amssymb}

\newtheorem{thm}{Theorem}

\newtheorem{lemma}{Lemma}
\newtheorem{lemma2}[lemma]{Lemma}
\newtheorem{lemma3}[lemma]{Lemma}

\newtheorem{lemma5}[lemma]{Lemma}
\newtheorem{thm2}[thm]{Theorem}

\newtheorem{prop}{Proposition}

\newtheorem{rmk1}{Remark}
\newtheorem{rmk2}[rmk1]{Remark}

\newcommand{\sig}{\tilde{\sigma}}
\newcommand{\N}{\mathcal{N}}
\begin{document}

\title{The smallest prime that does not split completely in a number field}

\author{Xiannan Li}
\email{xli@math.stanford.edu}
\subjclass[2000]{Primary 11N60; Secondary 11R42.}
\address{Department of Mathematics, Stanford University, Stanford, CA 94305}

\date{\today}
\thanks
{The author is partially supported by a NSERC PGS-D award.}
\maketitle
\begin{abstract}
We study the problem of bounding the least prime that does not split completely in a number field.  This is a generalization of the classic problem of bounding the least quadratic non-residue.  Here, we present two distinct approaches to this problem.  The first is by studying the behavior of the Dedekind zeta function of the number field near $1$, and the second by relating the problem to questions involving multiplicative functions.  We derive the best known bounds for this problem for all number fields with degree greater than $2$.  We also derive the best known upper bound for the residue of the Dedekind zeta function in the case where the degree is small compared to the discriminant.
\end{abstract}
\section{Introduction}
\subsection{Historical Background}
Let $\N$ denote the least quadratic non-residue modulo a prime $p$.  An old and difficult problem in number theory is to find good upper bounds for $\N$.  Much work has been done on this problem, and we will only mention a small selection of that here.  

The best result known arises from considerations of cancellation in character sums.  To be more specific, let $\chi$ be the quadratic character with modulus $p$.  Then we say that $\chi$ exhibits cancellation at $x=x(p)$ if $\sum_{n\leq x} \chi(n) = o(x)$.  Thus, the well known bound of Polya and Vinogradov for character sums implies that cancellation occurs for $x = p^{1/2+o(1)}$ (see \cite{Da}).  Vinogradov \cite{Vi} proved that such cancellation implies that the least quadratic non-residue is $\N \ll p^{\frac{1}{2\sqrt{e}}+o(1)}$.  Burgess showed in \cite{Bu} that cancellation occurs at $x = p^{1/4 + o(1)}$, and this implied that 
\begin{equation} \label{burgesseqn}
\N \ll p^{\frac{1}{4\sqrt{e}}+o(1)},
\end{equation}which apart from different quantifications of $o(1)$ is the best result known.

Vinogradov conjectured that $\N \ll_\epsilon p^{\epsilon}$ for any $\epsilon > 0$.  This is very reasonable since the Riemann hypothesis for $L(s, \chi)$ implies the stronger bound of 
\begin{equation} \label{eqnGRHnonquadres}
\N \ll \log^2 p.
\end{equation}  The true bound is suspected to be $\N \ll \log p \log\log p$, arising from probablistic considerations.

\subsection{Generalization} This problem is the same as finding the least prime which does not split completely in a quadratic field.  A generalization is to find upper bounds for the least prime which does not split completely in an arbitrary number field.  Let $K$ be a number field of degree $l$ with discriminant $d_K$, $\N$ the least prime which does not split, and let $\zeta_K(s)$ denote its Dedekind zeta function.  Then $\zeta_K(s)$ is analytic on the complex plane except for a simple pole at $s=1$.  We have moreover that the Euler product
$$\zeta_K(s) = \prod_\mathfrak{p} \left(1-\frac{1}{N(\mathfrak{p})^s}\right)^{-1}
$$holds for $\Re s > 1$, where the product is over prime ideals $\mathfrak{p}$ and $N(\mathfrak{p})$ denotes the norm of $\mathfrak{p}$.  We note that if all integer primes split over $K$ that the Euler product for $\zeta_K(s)$ would be the same as that for $\zeta(s)^l$, where as usual $\zeta(s)$ denotes the Riemann zeta function.  Since $\zeta(s)^l$ has a pole of order $l$ at $s=1$ and $\zeta_K(s)$ has only a simple pole at $s=1$, we see that not all primes split.  This also leads to quantifications of the statement that the least prime which does not split cannot be too large and even suggests that stronger results should be available as $l$ grows.  Using this approach, K. Murty \cite{Mu} showed that assuming GRH for $\zeta_K(s)$ that $\N \ll \left(\frac{\log d_K}{l-1}\right)^2$ which is analogous to (\ref{eqnGRHnonquadres}).  Unconditionally, Murty notes in a remark in \cite{Mu} that his method would give a bound with a main term that is of the form
\begin{equation}
\N \ll d_K^{\frac{1}{2(l - 1)}}.
\end{equation}
This type of result was explicitly proved later using essentially elementary methods by Vaaler and Voloch \cite{VaVo}.  Their result is that
\begin{equation}\label{VVresult}
\N \leq 26l^2 d_K^{\frac{1}{2(l - 1)}}.
\end{equation}provided that 
$$d_K \geq \frac{1}{8}e^{2(l-1)\max(105, 25\log^2l)}.$$  
Vaaler and Voloch note that this result is an improvement on the more general result of Lagarias, Montgomery, and Odlyzko\cite{LMO}.  The latter condition on the size of $d_K$ is artificial, and there is reason to expect even better results when $d_K$ is small compared to $l$.  

Can this result be improved by some generalization of Vinogradov's method?  Interestingly enough, we will show that this is not the case in Theorem \ref{thmnegativresult}.  In fact, the best result from Vinogradov's approach is also a bound of the same form.  Later in \S 2, Lemma \ref{splitlem} and the discussion immediately preceding gives an alternate fourth proof of the $d_K^{\frac{1}{2(l-1)}}$ bound.

It thus appears that $d_K^{\frac{1}{2(l-1)}}$ is a natural barrier.  However, using some ideas involving basic information on the zeros if $\zeta_K(s)$, we prove a result of the form
$$\N \ll d_K^{\frac{1}{4(l-1)}\left(1+o(1)\right)}
$$in Theorem \ref{main}.

We also show that approaching the problem with multiplicative functions does pay dividends in some cases, which appear in Theorems \ref{thmcubic} and \ref{thmbiquad} where we derive good bounds for $\N$ in the cases where $K$ is cubic or biquadratic.  The idea here is to study how certain multiplicative functions interact with one another and take advantage of the behaviour of extremal quadratic characters.  The behaviour of extremal quadratic characters has appeared previously in \cite{DMV}, which reproduces unpublished work of Heath-Brown.  It is also contained in unpublished work of Granville and Soundararajan \cite{GSBurgess}.  In Lemma \ref{lemquadchar}, we quantify what it means for a quadratic character to be almost extremal, which may be of independent interest.  

In the cubic case, a consideration of the multiplicative functions involved will immediately generate a bound of $\N\ll d_K^{\frac{1}{4\sqrt{e}}+\epsilon}$ where $4\sqrt{e} = 6.59...$.  By studying almost extremal quadratic characters, we will show a modest improvement of $\N \ll d_K^{\frac{1}{6.64}}$.  We also give the following simple example in the biquadratic case here. Given moduli $q_1$ and $q_2$ where for simplicity we assume that $q_1 \asymp q_2 \asymp q$ for some $q$, the least quadratic non-residue for either $q_1$ or $q_2$ is $\ll q^{\frac{1-\delta}{4\sqrt{e}}}$ for some $\delta>\frac{7}{100}$.

\subsection{On residues}
This discussion is related to another interesting problem - that of finding upper bounds on the residue $\kappa$ of $\zeta_K(s)$ at $s=1$.  We remind the reader that the class number formula relates $\kappa$ to various algebraic invariants of $K$.  Specifically, let $r_1$ and $2r_2$ denote the number of real and complex embeddings of $K$, $h$ the class number, $R$ the regulator, and $\omega$ the number of roots of unity.  Then,
$$\kappa = \frac{2^{r_1}(2\pi)^{r_2}hR}{\omega \sqrt{d_K}}.$$  The best known explict upper bound is due to Louboutin who in \cite{Lou2000} showed that 
\begin{equation}
\kappa \leq \left( \frac{e \log d_K}{2 (l-1)}\right)^{l-1}.
\end{equation}We also refer to \cite{Lou2000} and \cite{Lou2001} for applications and connections of this type of result to other questions as well as references to previous works from Siegel as well as Lavrik and Egorov.  We will show a result of the from
\begin{equation*}
\kappa \leq \left( \frac{(1+o(1))e^\gamma \log d_K}{4l}\right)^{l-1},
\end{equation*}when $\frac{l}{\log d_K} = o(1)$ is small, and where $\gamma = 0.577...$ is Euler's constant.  See Theorem \ref{residueupperthm} for the exact result.

\subsection{Statement of Results}
We consider these problems from two different vantage points.  The first is via analysis of $L$-functions attached to the number field $K$, and the other stems from Vinogradov's work and work on multiplicative functions as in \cite{GS}.  It is interesting that the latter method, which gives us the best known bounds in the quadratic case, is not optimal for number fields of large degree.  Indeed, the first method will give us the best known upper bounds on the least prime that does not split for number fields of large degree and will also lead to such a result on the residue of the Dedekind zeta function.  Specifically, we will show in \S 2 that

\begin{thm} \label{main}
Let $K$ be a number field of degree $l$ and discriminant $d_K$.  Let $\N$ be the least prime that does not split completely in $K$.  Then
$$\N \ll_\epsilon d_K^{\frac{1+\epsilon}{4A(l-1)}}.
$$Here $A = \sup_{\lambda\geq 0}\frac{1-\frac{l}{l-1}e^{-\lambda}}{\lambda}$ satifies $A \geq 1-\sqrt\frac{2}{l-1} = 1+O(\frac{1}{\sqrt{l}}) \rightarrow 1$ as $l\rightarrow \infty$.
\end{thm}

The dependence on $\epsilon$ may be quantified explicitly by
$$\N \ll \left(\frac{\log d_K}{l}\right)^2 d_K^{\frac{1+o(1)}{4A(l-1)}}.
$$
In the above, $o(1)$ denotes a quantity which tends to $0$ as either $l$ or $d_K$ grows.  It is illustrative here to consider two examples.  First, if we consider some sequence of number fields such that $d_K \leq C^l$ for some constant $C$, then we see that the least prime that does not split must be bounded by a constant.  This case does not appear in the previous work \cite{VaVo}.  Secondly, in the opposite case where $\frac{\log d_K}{l} \rightarrow  \infty$, then $\N \ll d_K^{\frac{1+o(1)}{4A}}$.  

\begin{rmk1}
The value of $A$ may be calculated for small $l$.  The result above beats the bound $d_K^{\frac{1}{2(l-1)}}$ when $l\geq 4$.  We comment that the best result in the case $l=2$ is still of the form (\ref{burgesseqn}).  The best result available in the case $l=3$ is also not of the form $d_K^{\frac{1}{2(l-1)}}$ but is the one described below in Theorem \ref{thmcubic}.

\end{rmk1}

Moreover, we also have the following upper bound for the residue of the Dedekind zeta function.
\begin{thm} \label{residueupperthm}
Let $\kappa$ be the residue at $s=1$ of the Dedekind zeta function of $K$, and let $d = \log d_K^{1/l}$.  Then
$$\kappa \ll \left(\frac{(\frac{1}{4}+B)e^{\gamma+\sqrt{\frac{2}{l}}}\log d_K}{l}\right)^{l-1}.
$$
where $B = \frac{2\log \log d}{\log d}+O\left(\frac{1}{\log d}\right)$. 
\end{thm}

In the case where $d_K$ grows faster than an exponential\footnote{By this, we mean that the statement $d_K \ll C^l$ is not true for any $C>0$.  An example would be the condition of Vaaler and Voloch immediately following (\ref{VVresult}).} in $l$, we have that $B = o(1)$.  Note also that since $d_K$ grows at least exponentially in $l$, $B$ is usually small.  However, the above result is not optimal for $d_K$ very small.  Rather, results like Hoffstein \cite{Hoffstein} and Bessassi \cite{Bessassi} optimize that particular case.  

\begin{rmk1}
The above results can be made explicit if desired but we choose not to do so for ease of exposition.  Improvements are possible in the coefficient in $B$ above as well as quantifications of the $\epsilon$ appearing in Theorem \ref{main}.

We also note that by applying a result of Stechkin \cite{Stech}, it is possible to prove the above results more explicitly, but replacing $\frac{1}{4}$ with $\frac{1-\frac{1}{\sqrt{5}}}{2} = 0.276... > 0.25 = 1/4$.  See Lemma 1 and environs for details.
\end{rmk1}

The utility of Vinogradov's method in the context of number fields has not been well understood.  We show in \S 3 that 

\begin{thm} \label{thmnegativresult}
Let $K$ be a number field of degree $l$ and discriminant $d_K$.  Let $f(n)$ be a real multiplicative function satisfying $0\leq f(p)\leq l$ on the primes and such that 
$$\sum_n \frac{f(n)}{n^s} = \zeta(s) \sum_n \frac{g(n)}{n^s},$$
valid for $\Re(s) >1$, for some multiplicative function $g(n)$ such that 
$$\sum_{n\leq x} g(n) = o(x)$$ for all $x>d_K^{1/2+o(1)}.$  
Then there exists some $p<d_K^{\frac{1}{2(l-1)}(1+O(l^{-1/2+\epsilon}))}$ such that $f(p)\neq l$.

Moreover, this is essentially the best possible result for large $l$.  To be specific, there exists a real multiplicative function satisfying all the properties above such that $f(p)=l$ for all $p<d_K^{\frac{1}{2(l-1)(1+O(l^{-1/2+\epsilon}))}}$.
\end{thm}

Thus, the technique behind Theorem \ref{main} is aware of information that can not be matched solely through the multiplicative functions approach, despite the fact that this approach gives the best known result for the quadratic case $l=2$.

However, the natural extension of Vinogradov's method and in particular, the structure in \cite{GS} has the advantage that it can utilize more information about the interaction between different multiplicative functions.  This allows us to improve bounds on $\N$ in the case of cubic and biquadratic fields.  Specifically, we will show in \S \ref{seccubicbi} that

\begin{thm2} \label{thmcubic}
Let notation be as in Theorem \ref{main}.  If $K$ is a cubic field, then we have that
$$\N \ll d_K^{\frac{1}{6.64}}.
$$
\end{thm2}  
A similar idea will enable us to show in \S \ref{seccubicbi} that
\begin{thm2} \label{thmbiquad}
Let $K$ be biquadratic with moduli $q_1$ and $q_2$.  Then we have that
$$\N \ll (q_1q_2)^{\frac{0.146}{2}}.
$$Furthermore, if $q_1\asymp q_2$, 
$$\N \ll (q_1q_2)^{\frac{0.141}{2}}.
$$
\end{thm2}
As we explain in \S \ref{seccubicbi}, these results should be compared to the trivial bounds of $d_K^{\frac{1}{4\sqrt{e}}+\epsilon}$ in the cubic case, and $(q_1q_2)^{\frac{1}{8\sqrt{e}}}$ in the biquadratic case.  Numerically, the results above are respectable, but have not been completely optimized.  We would like to exhibit that an interesting interaction between multiplicative functions leads to better bounds, rather than to push for the best possible numerical result.

\subsection{Notation}
In the following, when we write $f = O(g)$ or equivalently $f \ll g$ for functions $f$ and $g$, we shall mean that there exists a constant $C$ such that $|f| \leq C|g|$.  In the case where $g$ is a function of $\epsilon$ where as usual, $\epsilon$ denotes an arbitrary positive number, $C$ is allowed to depend on $\epsilon$.  Unless otherwise stated, $C$ is absolute, and in particular, $C$ never depends on the number field $K$.  We will also use $o(1)$ to denote a quantity which tends to $0$ as either $d_k \rightarrow \infty$ or $l \rightarrow \infty$ except in \S 3, where we are not concerned with uniformity in $l$ and $o(1)$ shall denote a quantity which tends to $0$ as $d_K \rightarrow \infty$ and $l/\log d_K \rightarrow 0$.   

\section{Working with the Dedekind zeta function}

As usual, write $s = \sigma + it$.  In this section, we will usually denote by $\rho = \beta + i\gamma $ a zero of the Dedekind zeta function.  Let 
$$F(s) = \Re \sum_\rho \frac{1}{s-\rho} = \sum_\rho\frac{\sigma - \beta}{(\sigma - \beta)^2 + (t-\gamma)^2},
$$defined for all $s\neq \rho$.  As before, let $l = r_1+2r_2$ denote the degree of $K$ over $\mathbb{Q}$ and $r_1$ and $2r_2$ be the number of real and complex embeddings of $K$ respectively.
Let $\xi_K(s) = s(s-1)\frac{d_K}{4^{r_2}\pi^l}^{s/2} \Gamma(s/2)^{r_1}\Gamma(s)^{r_2} \zeta_K(s)$.  Then $\xi_K(s)$ is entire of order $1$ and has a Hadamard product of the form 
$$\xi_K(s) = e^{A+Bs}\prod_\rho \left(1-\frac{s}{\rho}\right)e^{s/\rho}.
$$  
Logarithmically differentiating $\xi(s)$ gives that
\begin{equation} \label{logdiffeqn}
 F(s) = \Re \left(\frac{1}{2} \log \frac{d_K}{2^{2r_2}\pi^l} + \frac{\zeta_K'}{\zeta_K}(s) + G(s) + \frac{1}{s} + \frac{1}{s-1}\right),
\end{equation}
where
$$G(s) = \Re \left( \frac{r_1}{2} \frac{\Gamma '}{\Gamma} \left(\frac{s}{2}\right) + r_2 \frac{\Gamma '}{\Gamma}(s)\right).
$$
In the above we have used that $\Re B = -\Re \sum_\rho \frac{1}{\rho}$.  (See (11) on pg. 82 of \cite{Da} in the case of $\zeta(s)$.  The proof for the general case is the same.)
We have that $$-\frac{\zeta_K'}{\zeta_K}(s) = \sum_{n\geq 1} \frac{\Lambda_K(n)}{n^s},
$$where $\Lambda_K(n) = 0$ if $n$ is not a power of a prime, and $0\leq \Lambda_K(p^r) \leq l\log p$.  Rewritting (\ref{logdiffeqn}) for $s=\sigma>1$ gives that
\begin{equation} \label{basiceqn}
\sum_{n\geq 1} \frac{\Lambda_K(n)}{n^\sigma} = \frac{1}{2} \log \frac{d_K}{2^{2r_2}\pi^l} + \frac{1}{\sigma-1} - F(\sigma) + G(\sigma) +\frac{1}{\sigma}.
\end{equation}
Then $F(\sigma) > 0$ and $\frac{\zeta_K'}{\zeta_K}(\sigma) < 0$ for $\sigma >1$.    This observation led Stark to his lower bounds on discriminants as in \cite{Stark}, and this will be our starting point.  Indeed, if we use that $F(\sigma)>0$ and that $G(\sigma) < 0$ for $\sigma$ close to $1$, we would have that for $1<\sigma<\frac{5}{4}$,
\begin{equation} \label{2ndbasiceqn}
\sum_{n\geq 1} \frac{\Lambda_K(n)}{n^\sigma} \leq \frac{1}{2}\log \frac{d_K}{2^{2r_2}\pi^n} + \frac{1}{\sigma-1}+1.
\end{equation}    

Note that $\Lambda_K(n)$ is maximized when $n$ is a prime that splits completely in $K$, and the inequality above is a statement of the form that $\Lambda_K(n)$ cannot be too large for many $n$.  With some work, this already leads to a bound of the form $\N \ll d_K^{\frac{1}{2(l-1)}(1+O(\frac{1}{\sqrt{l}}))}$, which is similar to the results of Murty \cite{Mu} and Vaaler and Voloch \cite{VaVo}.  Specifically, the following lemma holds.

\begin{lemma} \label{splitlem}
Suppose that for some quantity $c>0$, the bound
\begin{equation}\label{lem1eqn}
\sum_{n\geq 1} \frac{\Lambda_K(n)}{n^\sigma} \leq c\log d_K + \frac{1}{\sigma-1}
\end{equation}holds for all $1+\frac{1}{\log d_K}\leq \sigma\leq 1+\frac{10\sqrt{l}}{\log d_K}$.  Also let
$$a(\lambda) = \frac{1-\frac{l}{l-1}e^{-\lambda}}{\lambda},
$$and let $A = \sup_{\lambda\geq 0}{a(\lambda)}$.
Then 
$$\N \ll d_K^{\frac{c}{A(l-1)}(1+o(1))}.
$$
\end{lemma}
\begin{proof}
If all primes split completely up to $x>2$, then $\Lambda_K(n) = l \Lambda(n)$ for all $n\leq x$ and where $\Lambda (n)$ is the usual von Mangoldt function.  Then by the Prime Number Theorem for $\mathbb{Q}\:$,
\begin{eqnarray*}
\sum_{n\leq x} \frac{\Lambda_K(n)}{n^\sigma} 
&=& \sum_{n\leq x} \frac{l\Lambda(n)}{n^\sigma} \\
&=&l\left(\int_1^{x}\frac{1}{t^\sigma}dt + O(1)\right) \\
&=& \frac{l}{\sigma-1}-\frac{lx^{1-\sigma}}{\sigma-1} + O(l).
\end{eqnarray*}  Thus we have from (\ref{lem1eqn}) that
$$\frac{l-1}{\sigma-1}-\frac{lx^{1-\sigma}}{\sigma-1} \leq c\log d_K + O(l)
$$
Set $\sigma = 1+\frac{\lambda}{\log x}$.  Then the above is the same as
$$\frac{(l-1)(\log x+O(1))}{\lambda}(1-\frac{l}{l-1}e^{-\lambda}) \leq c\log d_K + O(l).
$$We may assume that the $O(1) = o(\log x)$ since otherwise the result is obvious.  Thus rearranging the above, we have that
$$\log x \leq \frac{c+o(1)}{a(\lambda)(l-1)}\log d_K + O(1).
$$
We note that $a(\lambda)$ has a global maximum for $\lambda>0$.  If we let $A$ be that maximum, then the result follows immediately.
\end{proof}

A corresponding statement on upper bounds for $\kappa$ also results from considerations of this type.  This conforms to the intuition that in order to maximize $\kappa$, we should put as much weight as possible on the small primes in the sum in (\ref{2ndbasiceqn}).  In other words, the worst case senario is when all the small primes split.  To this end, we prove the following lemma.

\begin{lemma} \label{greedylem}
Assume that (\ref{lem1eqn}) holds as in Lemma 1 for some $\sigma =1+ \frac{\alpha}{\log d_K}$.  Assume that there exists some $T$ such that $l \sum_{n\leq T} \frac{\Lambda(n)}{n^\sigma} \geq c\log d_K + \frac{1}{\sigma-1}$.  Then,
$$\log \kappa \leq c\alpha + l\sum_{n\leq T}\frac{\Lambda(n)}{n^\sigma \log n} + \log (\sigma -1).
$$

\end{lemma}
\begin{proof}
We first show that
\begin{equation}\label{greedylemeqn}
\log \zeta_K(\sigma) \leq l \sum_{n\leq T}\frac{\Lambda(n)}{n^\sigma \log n}.
\end{equation}

To this end, let $S(t) = \sum_{n\leq t}\frac{\Lambda_K(n)}{n^\sigma}$, and $\tilde{S}(t) = \min \{\sum_{n\leq t}\frac{l\Lambda(n)}{n^\sigma}, c\log d_K + \frac{1}{\sigma-1}\}.$  Essentially $\tilde{S}(t)$ is the version of $S(t)$ which grows at the fastest rate possible, and visibly $S(t) \leq \tilde{S}(t)$.    Note also that $\tilde{S}(t) = c\log d_K + \frac{1}{\sigma-1}$ is constant for $t\geq T.$
Since $$\log \zeta_K(\sigma) = \sum_{n\geq 1}\frac{\Lambda_K(n)}{n^\sigma\log n},$$ by partial summation,
\begin{eqnarray*}
\log \zeta_K(\sigma) 
&=& \int_1^\infty \frac{S(t)}{t\log^2t}dt\\
&\leq& \int_1^\infty \frac{\tilde{S}(t)}{t\log^2t}dt\\
&=& \int_1^{T} \frac{\tilde{S}(t)}{t\log^2t}dt + \frac{\tilde{S}(T)}{\log T}\\
&=& l \sum_{n\leq T} \frac{\Lambda(n)}{n^\sigma \log n},
\end{eqnarray*}and this proves (\ref{greedylemeqn}).
By (\ref{lem1eqn}), we have by integration that
\begin{eqnarray*}
\log \kappa - \log (\sigma-1) \zeta_K(\sigma) 
&\leq& c(\sigma-1)\log d_K  \\
&=&c\alpha,
\end{eqnarray*}as desired.
\end{proof}

Again, since (\ref{lem1eqn}) follows from (\ref{2ndbasiceqn}) with $c=\frac{1}{2}$, with some work the lemma above gives us a bound roughly of the form $\kappa \ll \left(\frac{\left(1+o(1)\right)e^\gamma\log d_K}{2(l-1)}\right)^{l-1}$, at least when $d_K$ is large when compared to $l$.  This is already an improvement over Louboutin's result when $\frac{l}{\log d_K}$ is small.  

It is clear from both Lemma \ref{splitlem} and \ref{greedylem} that we gain information on both $\N$ and $\kappa$ if we were able to extract non-trivial contribution from $F(\sigma)$ in (\ref{basiceqn}).  However, the discussion immediately following (\ref{basiceqn}) neglected the contribution of the zeros entirely. We now proceed to rectify that situation. There are a number of possible approaches to this, and the best seems to be due to Heath-Brown \cite{DRHB} in the case of the Dirichlet $L$-functions.  There are some minor technicalities in our case, which we resolve with the help of the following Lemma.

\begin{lemma2}\label{technicallem}
Let $\sigma_0 > 1+\frac{1}{\log d_K}$ and $\frac{1}{4}<R<\frac{1}{2}$.  Let $C_1$ be the half circle of radius $R$ centered at $\sigma_0$ with real part to the right of $\sigma_0$.  Let $\mathcal{D} = \log \frac{\log d_K}{l}$  Then 
$$\frac{1}{\pi R} \int_{C_1} |\log (s-1)\zeta_K(s)|ds \leq l\mathcal{D} + O(l).
$$
\end{lemma2}
\begin{proof}
Let $s = \sigma + it$ where $\sigma>1+\frac{1}{\log d_K}$.  Then 
$$|\log \zeta_K(s)| \leq |\log \zeta_K(\sigma)| \leq \log \zeta_K\left(1+\frac{1}{\log d_K}\right).
$$These inequalities follow upon comparing Dirichlet series and since the coefficients of $\log \zeta_K(\sigma)$ are positive.  We now claim that 
$$\log \zeta_K\left(1+\frac{1}{\log d_K}\right) \leq l\mathcal{D} + O(l).
$$

Our calculations in Lemma \ref{greedylem} gives us this bound almost immediately.  Specifically, we have from (\ref{2ndbasiceqn}) and (\ref{greedylemeqn}) in the proof of Lemma \ref{greedylem} that
$$\log \zeta_K(\sigma) \leq l\sum_{n\leq T}\frac{\Lambda(n)}{n^\sigma \log n}
$$provided that $l\sum_{n\leq T}\frac{\Lambda(n)}{n^\sigma} \geq \frac{\log d_K}{2} + \frac{1}{\sigma-1} + 1$.  Say that $\sigma = 1+\frac{1}{\log d_K}$.  Then \begin{eqnarray*}
l\sum_{n\leq T}\frac{\Lambda(n)}{n^\sigma} 
&\ge& le^{-\frac{\log T}{\log d_K}} \sum_{n\leq T}\frac{\Lambda(n)}{n}\\
&\ge& le^{-\frac{\log T}{\log d_K}} \log T + O(l).
\end{eqnarray*} Thus there is some constant\footnote{Later on, we will have a specific value of $C$ when we prove the Theorem \ref{residueupperthm}, but for our present purposes, it suffices to note that this is possible for some absolute constant $C$.} $C$ such that $l\sum_{n\leq T}\frac{\Lambda(n)}{n^\sigma} \geq \frac{\log d_K}{2} + \frac{1}{\sigma-1}$ for $T = d_K^{\frac{C}{l}}$.  Then we must have that $\log \zeta_K(\sigma) \leq l \log\log T + O(l) = l\log \frac{C\log d_K}{l} + O(l) = l \log \frac{\log d_K}{l} + O(l)$.  Note that our bounds here hold uniformly in $d_K$ and $l$.
\end{proof}

Now we are ready to prove the following.
\begin{lemma3} \label{upperboundlem}
For any $1+\frac{1}{\log d_K} < \sigma_0 \leq 1+\frac{10\sqrt{l}}{\log d_K}$ and $\mathcal{D} = \log \frac{\log d_K}{l}$, we have that
$$-\frac{\zeta_K '}{\zeta_K}(\sigma_0) \leq \left(\frac{1}{4}+o(1)\right)\log d_K + \frac{1}{\sigma_0-1} + 2l \mathcal{D} + O(l).
$$uniformly in $\sigma_0$.
\end{lemma3}
\begin{proof}
Let $f(s) = (s-1)\zeta_K(s)$.  Let $C_R$ denote the circle of radius $R$ with center $\sigma_0$ with no zeros of $f(s)$ on $C_R$.  Then $f(s)$ is analytic and we may apply Lemma 3.2 in \cite{DRHB} to get that
$$-\Re \frac{f'}{f}(\sigma_0) = {\sum_\rho}'\left(\frac{1}{\sigma_0-\rho} - \frac{\sigma_0-\rho}{R^2}\right) -\frac{1}{\pi R} \int_0^{2\pi} \cos \theta \log|f(\sigma_0+Re^{i\theta})|d\theta
$$where $\sum '$ denotes a sum over all zeros of $f$ within $C_R$.  This is related to Jensen's formula and we refer the reader to \cite{DRHB} for a proof.  

We now need to bound the integral above, which we split into two ranges.  The first is when $0\leq \theta \leq \pi/2$ and $3\pi/2\leq \theta\leq 2\pi$.  The second is when $\pi/2\leq \theta\leq 3\pi/2$.  In the first range Lemma \ref{technicallem} tells us that $\frac{1}{\pi R}\int_{C_1} |\log|f(\sigma_0+Re^{i\theta})| \leq l\mathcal{D} + O(l)$.

In the second range, we use the convexity bound $\zeta_K(\sigma + it)\ll d_K^{\frac{1-\sigma}{2}}e^{l\mathcal{D}+Cl}$ for some $C>0$.  Since $\cos \theta \leq 0$, we have that
\begin{eqnarray*}
\cos \theta \log|f(\sigma_0+Re^{i\theta})| 
&\geq& \cos \theta \frac{1-\sigma_0 -R\cos \theta}{2}\log d_K(1+o(1)) \\
&+& \cos \theta(l\mathcal{D} + Cl)\\
&\geq& -\cos \theta\frac{R\cos \theta}{2}\log d_K (1+o(1))+\cos \theta(l\mathcal{D} + Cl).
\end{eqnarray*}
In the above, we have used that $1-\sigma_0 = o(1)$.  Now, we may assume that $2/\pi<R<1$ so the contribution of the second term to the integral is $\leq l\mathcal{D} + Cl$.

The contribution of the first term to the integral is
$$\leq \left(\frac {\log d_K+o(1)}{2\pi R}\right) \left(\int_{\pi/2}^{3\pi/2}R\cos^2 \theta d\theta + o(1)\right)
=  \left(\frac{1}{4}+o(1)\right)\log d_K.
$$Hence
\begin{eqnarray*}
-\frac{\zeta_K '}{\zeta_K}(\sigma_0) 
&\leq& \frac{1}{\sigma_0-1} -\Re{\sum_\rho}'\left(\frac{1}{\sigma_0-\rho} - \frac{\sigma_0-\rho}{R^2}\right) + \frac{1+o(1)}{4}\log d_K \\
&+& 2l\mathcal{D} + O(l)\\
&\leq& \frac{1}{\sigma_0-1} + \frac{1+o(1)}{4}\log d_K+2l\mathcal{D} + O(l),
\end{eqnarray*}where we have used that
$$\Re \left(\frac{1}{\sigma_0-\rho} - \frac{\sigma_0-\rho}{R^2}\right)
= (\sigma_0-\beta)\left(\frac{1}{|\sigma_0-\rho|^2}-\frac{1}{R^2}\right) \geq 0.
$$
\end{proof}

\subsection{Proof of Theorem \ref{main}}
Theorem \ref{main} now follows immediately from Lemma \ref{splitlem} and Lemma \ref{upperboundlem} with $c = \frac{1}{4} + o(1) + \frac{2l\mathcal{D}+O(l)}{\log d_K}$ where $\mathcal{D} = \log \frac{\log d_K}{l}$ as before.  We have that for $d = d_K^{\frac{1}{l}}$ that
$$\frac{2l\mathcal{D}+O(l)}{\log d_K} = 2\frac{\log \log d}{\log d}+O\left(\frac{1}{\log d}\right)
$$Also
$$d_K^{\frac{2\log \log d+O(1)}{(l-1)\log d}} \ll (\log d)^2.
$$

We further need to verify that $A = \sup_{\lambda\geq 0}{\frac{1-\frac{l}{l-1}e^{-\lambda}}{\lambda}} \geq 1-\sqrt{\frac{2}{l-1}}$.  We have that
\begin{eqnarray*}
\frac{1-\frac{l}{l-1}e^{-\lambda}}{\lambda}
&=& \frac{1-e^{-\lambda}}{\lambda} - \frac{e^{-\lambda}}{(l-1)\lambda}\\
&\geq& 1-\frac{\lambda}{2} - \frac{1}{(l-1)\lambda}\\
&=&1-\sqrt{\frac{2}{l-1}},
\end{eqnarray*}upon setting $\lambda = \sqrt{\frac{2}{l-1}}$.

\subsection{Proof of Theorem \ref{residueupperthm}}
It remains to prove the upper bound on the residue $\kappa$ in Theorem \ref{residueupperthm}.  As before, set $d =\log d_K^{1/l}.$  We already have from Lemma \ref{greedylem} that with $\sigma = 1+\frac{\alpha}{\log d_K}$ and for any $T$ such that $l \sum_{n\leq T} \frac{\Lambda(n)}{n^\sigma} \geq c\log d_K + \frac{1}{\sigma-1}$ with $c = \frac{1}{4}+\frac{2\log \log d}{\log d}+O(\frac{1}{\log d})+o(1)$, then
\begin{eqnarray*}
\log \kappa &\leq& c\alpha + l\sum_{n\leq T}\frac{\Lambda(n)}{n^\sigma \log n} + \log (\sigma -1) \\
&\leq& c\alpha + \log (\sigma -1)+ l(\log \log T + \gamma + \frac{2}{\log^2 T}),
\end{eqnarray*}where the latter line follows from taking logarithms of (3.27) of \cite{RosScho}.  Set $\alpha = 4\sqrt{l}$ and recall that that $\sigma = 1+\frac{\alpha}{\log d_K}$.
We need to find the smallest admissible value of $T$.  Let $S(x) = \sum_{n\leq x}\frac{\Lambda(n)}{n} = \log x -C + E(x)$ for some constant $C$.  From \cite{RosScho}, we know $-\frac{1}{\log x} < E(x)   < \frac{1}{\log x}$.  We have that
\begin{eqnarray*}
\sum_{n\leq T}\frac{\Lambda(n)}{n^\sigma}
&=& \int_{2^-}^T \frac{1}{x^{\sigma-1}}d(S(x))\\
&=& \int_{2^-}^T \frac{1}{x^{\sigma}}dx + \frac{E(T)}{T^{\sigma-1}}\\
&=& \frac{1}{\sigma -1}(2^{\sigma-1} - T^{\sigma -1}) + \frac{E(T)}{T^{\sigma-1}}\\
&=& \log T + \frac{E(T)}{T^{\sigma-1}} + O((\sigma-1)T^{\sigma-1})
\end{eqnarray*} 
We see easily that $T\ll d_K^{1/l}$ so $(\sigma-1)T^{\sigma-1} = o(1)$.  We thus have that
$$\log T = \frac{\log d_K}{l}(c+\frac{1}{\alpha}) + R(T),
$$where $|R(T)| < \frac{1}{\log T}$.  If $\log T \geq \frac{\log d_K}{4l}$, we may absorb $R(T)$ into the $O(\frac{l}{\log d_K})$ term inside $c$ so that we write
$$\log T = \frac{\log d_K}{l}(c+\frac{1}{\alpha}).
$$Otherwise, $\log T\leq \frac{\log d_K}{4l}\leq \frac{\log d_K}{l}(c+\frac{1}{\alpha})$.

Either way, we have
\begin{eqnarray*}
\kappa &\leq& \exp(\sqrt{l}) \frac{4\sqrt{l}}{\log d_K}(e^\gamma \log T)^{l}\\
&\leq&  \frac{4e^\gamma c}{\sqrt{l}} \left(ce^{\gamma + \frac{2}{\sqrt{l}}} \frac{\log d_K}{l}\right)^{l-1},
\end{eqnarray*}where we have written $c+\frac{1}{\alpha} \leq ce^{\frac{1}{c\alpha}}$.  Let $B = \frac{2\log \log d}{\log d} + O(\frac{l}{\log d})$.  Then we have also that
$$\kappa \ll \left(\left(\frac{1}{4}+B\right)e^{\gamma + \frac{2}{\sqrt{l}}} \frac{\log d_K}{l}\right)^{l-1}
$$

Since $d_K$ grows at least as fast as an exponential in $l$, $B$ is always bounded.  As mentioned before, we are most interested here in the case when $d$ grows, so that $B = o(1)$.

\section{On multiplicative functions}
\subsection{Preliminaries}
Let $\zeta_K(s) = \sum_{n\geq 1}\frac{a(n)}{n^s}$ be the Dirichlet series for $\zeta_K(s)$.  For this section, set $f(n)$ to be the multiplicative function such that 
$$\frac{\zeta_K(s)}{\zeta(s)} = \zeta_K(s) \prod_p \left(1-\frac{1}{p^s}\right) = \sum_n \frac{f(n)}{n^s},
$$for $\Re s > 1$.  Note that at primes, $f(p) = a(p) - 1$.  We first note that $f(n)$ exhibits cancellation at $d_K^{1/2+o(1)}$.  This argument is a standard on wherein we examine the Dirichlet series $D(s) := \frac{\zeta_K(s)}{\zeta(s)} = \sum_{n\geq 1} \frac{f(n)}{n^s}$.  Then the standard zero free region for $\zeta(s)$ is sufficient to find cancellation using Perron's formula.  

The question of bounding the least non-split prime can be converted to a more general question involving $f(n)$.  To be precise, knowing that $f(n)$ exhibits cancellation at $d_K^{1/2+o(1)}$, what is the maximum $y$ such that $f(p) = l-1$ for all $p\leq y$?  

We now collect some facts about multiplicative functions which will be useful for the remainder of this section.  Since the applications will be towards proving Theorems \ref{thmnegativresult}, \ref{thmcubic} and \ref{thmbiquad}, we will not take the same care to prove uniformity in $l$ as in the previous results.  The following are essentially culled from the work of Granville and Soundararajan \cite{GS}.  Granville and Soundararajan proved their results for the case where $|f(n)| \leq 1$, but the proofs extend to our case with very minor modifications.  We summarize the results, and the required modifications to the proofs below.

Let $f(n)$ be the multiplicative function defined above with $-1\leq f(p) \leq k$.  Here, $k=l-1$ where $l$ is the degree of our number field $K$.  Fix some $y\geq 2$ such that $f(p)=k$ for all $p\leq y$.  Note that this implies that all $y$ smooth numbers $n$ satisfy $f(n) = d_k(n)$, where the latter is the number of ways of writing $n$ as a product of $k$ numbers.  Then define
$$\sigma(u) = \frac{1}{y^u \log^{k-1} y}\sum_{n\leq y^u} f(n),
$$and
$$P(u) = \frac{1}{y^u}\sum_{p\leq y^u} f(p)\log p.
$$Then there are two related ways to express the relationship between $\sigma(u)$ and $P(u)$.  

First say that $\sig$ satisfies the convolution equation
\begin{equation}\label{eqnconvolution}
u\sig(u) = \sig * P(u) = \int_0^u \sig(u-t)P(t)dt,
\end{equation}
for $u>1$ subject to $\sig(u) = u^{k-1}$ for $u\leq 1$.  Then for our case, we will have that $\sig(u) = \sigma(u) + o(1)$.  The proof of this when $|f(n)|\leq 1$ is contained in \S 4 of \cite{GS}, and the proof for our case is almost the same.  In the proof of Proposition 4.1 in \cite{GS}, we define the multiplicative function $g(n)$ satisfying $g(p^k) = f(p^k) - f(p^{k-1})$ for all prime powers.  The non-negative function $|g(n)|$ still satisfies the hypothesis of Theorem 2 in \cite{HR}, which give that
$$\sum_{n\leq x} |g(n)| \leq k \frac{x}{\log x} \sum_{n\leq x}\frac{|g(n)|}{n}\left(1+O\left(\frac{1}{\log x}
\right)\right).
$$  The only modification in the proofs thereafter would be to replace error terms of the form $O(A)$ by $O(kA)$.  

Next, we also have a inclusion exclusion relationship.  To be specific, let 
\begin{equation}\label{eqnintegrals}
I_j(u) = \int_{\substack{t_1+...+t_j \leq u\\ t_i\geq 1}} \left(\frac{u-\sum_{i=1}^jt_i}{u}\right)^{k-1}\prod_{i=1}^j \frac{k-P(t_i)}{t_i}dt_1...dt_j.
\end{equation}  Then
\begin{equation}\label{eqninclusionexclusion}
\sig(u) = u^{k-1}\sum_{j=0}^\infty \frac{(-1)^j}{j!}I_j(u),
\end{equation}
where we set $I_0 = 1$.  The above sum is finite since $I_j(u) = 0$ for $u\leq j$.  We digress briefly to elucidate this inclusion exclusion relationship.  We have that $\sig(u) \leq u^{k-1}+o(1)$ since $f(n) \leq d_k(n)$.  Now, if $p\geq y$, note that $\sum_{\substack{n\leq y^u\\p|n\\p\geq y}} 1 = \left(\sum_{\substack{n\leq y^u\\p|n\\p^2\not|\: n\\p\geq y}}1\right)  (1+O(1/y))$, so
\begin{eqnarray*}
\sum_{n\leq y^u}f(n)
&\geq& y^u \log^{k-1} (y^u)(1+o(1))-\sum_{y\leq p\leq y^u} \sum_{\substack{n\leq y^u\\p|n}} (d_k(n) - f(n))\\
&\geq& y^u \log^{k-1} (y^u)(1+o(1))-\sum_{y\leq p\leq y^u} \sum_{m\leq y^u/p} d_k(m)(k-f(p))\\
&\geq& y^u \log^{k-1} (y^u)(1+o(1))-\sum_{y\leq p\leq y^u} (k-f(p))\frac{y^u}{p}\left(\log \frac{y^u}{p}\right)^{k-1}.
\end{eqnarray*}

An appropriate application of summation by parts brings this to $\sigma(u) \geq u^{k-1} (1-I(1)+o(1))$, and one can derive (\ref{eqninclusionexclusion}) in this manner.  However, we will relate this independently to the convolution equation (\ref{eqnconvolution}).  

For fixed $P(t)$, note that the solution $\sig(u)$ to (\ref{eqnconvolution}) is unique by the same proof as Theorem 3.3 in \cite{GS}.  Thus to prove (\ref{eqninclusionexclusion}), it suffices to show that $u^{k-1}\sum_{j=0}^\infty \frac{(-1)^j}{j!}I_j(u)$ satisfies the convolution equation (\ref{eqnconvolution}).  The calculation here is similar to Lemma 3.2 in \cite{GS} and the main step is checking that 
\begin{equation}\label{eqnmainstepinclusionexclusion}
k*J_j(u) = uJ_j(u) - j((k-P)*J_{j-1})(u),
\end{equation}
where $J_j(u) = u^{k-1}I_j(u)$.  This is because (\ref{eqnmainstepinclusionexclusion}) immediately implies that 
$$u\sum_{j=1}^\infty \frac{(-1)^j}{j!}J_j(u) + u^{k} = u^{k} + k* \sum_{j=1}^\infty \frac{(-1)^j}{j!}J_j(u) - \sum_{j=0}^\infty \frac{(-1)^j}{j!}((k-P)*J_{j-1})(u)$$  
which becomes (\ref{eqninclusionexclusion}) upon noting that $u^k = k*J_0$.  

Some of the details in proving (\ref{eqnmainstepinclusionexclusion}) differ slightly from that in \cite{GS} so we will provide the proof in the Lemma below.

\begin{lemma5}\label{leminclusionexclusion}
For $J_j(u)$ defined as above, $k*J_j(u) = uI_j(u) - j((k-P)*J_{j-1})(u)$.
\end{lemma5}

\begin{proof}
For notational convenience, set $S = \sum_{i=1}^jt_i$.  Then
\begin{eqnarray*}
k*J_j(u) 
&=& \int_0^uk\int_{\substack{S \leq t\\ t_i\geq 1}} \left(t-S\right)^{k-1}\prod_{i=1}^j \frac{k-P(t_i)}{t_i}dt_1...dt_jdt\\
&=& \int_{\substack{S \leq u\\ t_i\geq 1}} \prod_{i=1}^j \frac{k-P(t_i)}{t_i}\int_S^u k\left(t-S\right)^{k-1} dt dt_1...dt_j\\
&=& \int_{\substack{S \leq u\\ t_i\geq 1}} \prod_{i=1}^j \frac{k-P(t_i)}{t_i}\left(u-S\right)^{k-1}\left(u-S\right) dt_1...dt_j\\
&=& uJ_j(u) - j \int_{\substack{t_1+...+t_{j-1} \leq u-t_j \leq u\\ t_i\geq 1}} t_j \times \\
&&\prod_{i=1}^j \frac{k-P(t_i)}{t_i}\left(u-t_j - \sum_{i=1}^{j-1}t_i\right)^{k-1} dt_1...dt_j\\
&=& uJ_j(u) - j (k-P)*J_{j-1}(u).
\end{eqnarray*}
\end{proof}

Henceforth, by an abuse of notation, we write $\sigma(u)$ for $\sig(u)$ as well, and suppress the $o(1)$ error.  Frequently, it will be useful to know that the minimal value of $P(t)$ gives the earliest cancellation in $\sigma(t)$.  The following Proposition tells us this.  For an alternate proof, see also Lemma 3.4 of \cite{GS}.  

\begin{prop}\label{propstructure}
Suppose that we have two multiplicative functions $f$ and $f^\sharp$.  Let $P(u) = \frac{1}{y^u}\sum_{p\leq y^u}f(p)\log p$ and $P^\sharp(u)=\frac{1}{y^u}\sum_{p\leq y^u}f^\sharp(p)\log p$.  Define $\sigma(u)$ and $\sigma^\sharp(u)$ to be the solutions to (\ref{eqnconvolution}) for $P(u)$ and $P^\sharp(u)$ respectively.  Further suppose that $P(u) = P^\sharp(u)$ for $u\leq 1$, and that $P(u) \leq P^\sharp(u)$ always.  Let $u_0$ be the first zero of $\sigma(u)$.  Then for $u\leq u_0$, $0\leq \sigma(u) \leq \sigma^\sharp(u)$.  
\end{prop}
\begin{proof}
We use $I_j(u)$ and $I_j^\sharp(u)$ to denote the various integrals defined as in (\ref{eqnintegrals}).  Further let $1_{(a, a+\epsilon)}(t)$ denotes the indicator function of the small interval $(a, a+\epsilon)$.  Without loss of generality, it suffices to prove the result in the case where $P^\sharp(t) = P(t) + \delta 1_{(a, a+\epsilon)}(t)$ for all $\delta>0$, all $a>1$ and $\epsilon$ arbitrarily small.  This is because linear combinations of functions of the form $\delta 1_{(a, a+\epsilon)}(t)$ are $L^2$ dense.  For notational convenience, set $S(t, u) = S(t) = \frac{(k-P(t))}{t}$ and $Q(t, u) = Q(t) =  \frac{\delta 1_{(a, a+\epsilon)}(t)}{t}$.  We may also assume that $u>1+a$ since otherwise $\sigma(u) = \sigma^\sharp(u)$.  Now fix some $1+a<u<u_0$, and say that $N \geq u$ is the smallest such integer.  We have that 
\begin{eqnarray*}
\textstyle
&&\sigma^\sharp(u) - \sigma(u)\\
&=&u^{k-1}\sum_{j=0}^N \frac{(-1)^j}{j!} (I_j^\sharp(u) - I_j(u))\\
&=&\sum_{j=1}^N \frac{(-1)^j}{j!} \int_{\substack{t_1+...+t_j \leq u\\ t_i\geq 1}} \left(u-\sum_{i=1}^jt_i\right)^{k-1} \\
&&\times\left(\prod_{i=1}^j \left(S(t_i)-Q(t_i)\right) - \prod_{i=1}^j \left(S(t_i)\right)\right) dt_1...dt_j\\
&=&\sum_{j=1}^N \frac{(-1)^{j-1}}{(j-1)!}\left(\mathcal{T}_j + 
+O(\epsilon^2)\right)
\end{eqnarray*}where $\mathcal{T}_j = \int_{\substack{t_1+...+t_j \leq u\\ t_i\geq 1}} Q(t_1)\left(u-\sum_{i=1}^jt_i\right)^{k-1} \prod_{i=2}^j S(t_i)dt_1...dt_j$.  Here, we have used that integrals containing two factors of $Q$ like
$$\int_{\substack{t_1+...+t_j \leq u\\ t_i\geq 1}} Q(t_1)Q(t_2)\prod_{i=3}^j S(t_i)dt_1...dt_j$$ are $O(\epsilon^2)$.  The terms containing one factor of $Q$ are the same by symmetry.  We now note that
\begin{eqnarray*}
&&\mathcal{T}_j\\
&=&\int_a^{a+\epsilon}Q(t_1) \left(u-\sum_{i=1}^jt_i\right)^{k-1}\int_{\substack{t_1+...+t_j \leq u\\ t_i\geq 1}} \prod_{i=2}^j S(t_i)dt_1...dt_j\\
&=&\int_a^{a+\epsilon}Q(t_1) dt_1\int_{\substack{t_2+...+t_j \leq u-a\\ t_i\geq 1}} \left(u-a-\sum_{i=2}^jt_i\right)^{k-1} \prod_{i=2}^j S(t_i)dt_2...dt_j\\
&+&O(\epsilon)\int_a^{a+\epsilon}Q(t_1) dt_1\\
&=&\int_a^{a+\epsilon}Q(t_1) dt_1 \left( u^{k-1}I_{j-1}(u-a) +O(\epsilon)\right)
\end{eqnarray*}
In the above, the $O(\epsilon)$ arises from replacing instances of $t_1$ by $a$ and using that $a\leq t_1\leq a+\epsilon$.  Combining the above with the previous equation gives us that
$$\sigma^\sharp(u) - \sigma(u) = \int_a^{a+\epsilon}Q(t_1) dt_1 \left(\frac{u^{k-1}}{(u-a)^{k-1}}\sigma(u-a) + O(\epsilon)\right).
$$If we pick $\epsilon$ to be sufficiently small, the latter is positive since $\int_a^{a+\epsilon}Q(t_1) dt_1>0$ and $\sigma(u-a)>0$.
\end{proof}
\begin{rmk2}
Actually, wherever we use this result, we have that $f^\sharp(p) \geq f(p)$.  When this is true, there is an alternative argument which we now sketch.  Let $g(n)$ be the multiplicative function defined by $f^\sharp  = f*g$, that is $f^\sharp(n) = \sum_{d|n}f(d)g(n/d)$.  Then since $f^\sharp(p) = f(p)+g(p)$, we must have $g(p)\geq 0$.  Hence $\sum_{n\leq x}f^\sharp(n) = \sum_{n\leq x}\sum_{d|n}f(d)g(n/d) = \sum_{d\leq x}f(d) \sum_{n\leq x/d}g(n)$.  One may then argue that the contribution from values of $g$ on the prime powers is benign and so the latter is an upper bound for $\sum_{n\leq x}f(n)$.  
\end{rmk2}

\subsection{Generalization of Vinogradov's method}
By Proposition \ref{propstructure}, we only need consider the case where $P(u) = k$ for $u\leq 1$, and $P(u) = -1$ otherwise.

By the convolution equation (\ref{eqnconvolution}), we get that $\sigma(u)$ satisfies the following differential difference equation:
\begin{equation}\label{diffsigeqn}
u\sigma '(u) + (1-k)\sigma(u)+(k+1)\sigma(u-1) = 0.
\end{equation}

\begin{lemma5} \label{lembasizero}
Say that $u_0$ is a zero of $\sigma(u)$.  Then $u_0 \gg k/\log k$.  
\end{lemma5}
\begin{proof}
Without loss of generality, we may suppose that $u_0$ is minimal.  By a change of variables $\tau(u) = \sigma(u)u^{1-k}$, we derive from (\ref{diffsigeqn}) that 
$$\tau '(u) = -(k+1)\left(1-\frac{1}{u}\right)^k \tau(u-1).
$$We see immediately that $\tau$ is positive and decreasing on $[0, u_0)$ so $-\tau '(u) \leq (k+1)\left(1-\frac{1}{u}\right)^k \ll (k+1) e^{-k/u}$.   The result follows since by mean value theorem, $1 \ll (u_0-1) (k+1) e^{-k/u}$ for some $u\in[1, u_0)$.
\end{proof}

This allows us to say that cancellation occurs later than $k/\log k$ but we require finer analysis in order to obtain that it must occur very near $k$.  For this, we use the saddlepoint method.

\subsection{The saddlepoint method}
Let $\hat{\sigma}(s) = \int_0^\infty \sigma(t) e^{-st} dt$ denote the Laplace transform of $\sigma(t)$.  In Lemma \ref{Latranslem} below, we will show that $\hat{\sigma}(s)$ can be analytically continued to all of $\mathbb{C}$.  Thus, by Laplace inversion
\begin{equation}\label{laplaceinverseeqn}
\sigma(u) = \frac{1}{2\pi}\int_{-\infty}^\infty \hat{\sigma}(s) e^{us} dt
\end{equation}
 where $s=x+it$ for fixed $x$.  The idea of the saddlepoint method is that the integral for $\sigma(u)$ above is dominated by a small interval where the argument of the integrand changes slowly.  First, we need to obtain a workable form for $\hat{\sigma}(s)$.  Our approach will be similar to the analysis of the classic Dickman's function in \S 5.4 of \cite{Ten}.

\begin{lemma5}\label{Latranslem}
\begin{eqnarray*}
\hat{\sigma}(s) &=& (k-1)! s e^{(k+1)(I(-s)+\gamma) } \\
&=& \frac{(k-1)!}{s^k} e^{-(k+1)J(s)},
\end{eqnarray*}
where $I(s) = \int_0^s \frac{e^t-1}{t}dt$, $J(s) = \int_0^{\infty} \frac{e^{-(s+t)}}{s+t}dt$ and $\gamma$ is Euler's constant.
\end{lemma5}
Note that $J(s)$ only has holomorphic extension to $\mathbb{C}\;\backslash (-\infty, 0]$ and the purpose of writing $\hat{\sigma}(s)$ in terms of $I(-s)$ is to analytically continue the transform to all of $\mathbb{C}$\;.
\begin{proof}
Note that a change of variables $t=v/s$ in the definition of the Laplace transform gives us that
$$s\hat{\sigma}(s) = \int_0^\infty e^{-v}\sigma(v/s)dv.$$
Differentiating both sides with respect to $s$ gives
\begin{eqnarray*}
\frac{d}{ds} s\hat{\sigma}(s) 
&=& \frac{1}{s} \int_0^\infty e^{-v}(-(v/s)\sigma '(v/s))dv\\
&=& \frac{1}{s} \int_0^\infty e^{-v}\left((k+1)\sigma \left(\frac{v}{s} -1\right)-(k-1)\sigma(v/s) \right)dv\\
&=& (k+1)e^{-s}\hat{\sigma}(s)-(k-1)\hat{\sigma}(s),
\end{eqnarray*}upon changing variables again.
Solving the differential equation above for $s\hat{\sigma}(s)$ gives
$$s\hat{\sigma}(s) = C\frac{e^{-(k+1)J(s)}}{s^{k-1}},
$$for some constant $C$.  We have that $\lim_{s\rightarrow \infty}J(s) = 0$ so
$$\lim_{s\rightarrow \infty} s^k \hat{\sigma}(s) = C.
$$On the other hand,
\begin{eqnarray*}
\lim_{s\rightarrow \infty} s^k \hat{\sigma}(s)
&=& \lim_{s\rightarrow \infty} s^{k-1} \int_0^\infty e^{-v}\left(\frac{v}{s}\right)^{k-1}dv\\
&=& \int_0^\infty e^{-v}v^{k-1}dv\\
&=& (k-1)!
\end{eqnarray*} from which it follows that $C = (k-1)!$.  Note that the first line follows from the fact that $\sigma(t) = t^{k-1}$ for $t\leq 1$, and that $e^{-v}$ decreases rapidly.  

By Lemma 7.1 of \S 5.4 of \cite{Ten}, we have that for $s \in \mathbb{C}\;\backslash (-\infty, 0]$ that 
$$-J(s) = I(-s)+\gamma +\log s,
$$and this concludes the proof.
\end{proof}

In order to apply the saddlepoint method, we first collect some information on the extrema of the integrand in (\ref{laplaceinverseeqn}).  In the sequel, we let $W(x)$ denote the Lambert W function which is defined by $x= W(x)e^{W(x)}$. We remind the reader that there exists two real branches of $W(x)$ when $x\geq -1/e$ which we denote by $W_0$ and $W_{-1}$, where they are distiguished by $W_0(0) = 0$ and $W_{-1}(0) = -\infty$.  

\begin{lemma5}\label{extremalem}
Let $\Phi(s) = \hat{\sigma}(s) e^{us}$ and let $\xi(u) = -W\left(\frac{-(k+1) e^{-k/u}}{u}\right)-k/u$, where $W$ is a branch of the Lambert W function.  Then $\Phi '(\xi) = 0$.  If $|u-k| \geq 2\sqrt{k}$, then we may pick $\xi(u)$ to be real.  In particular, we pick 
\begin{equation}
\xi(u) = 
\begin{cases}
-W_0\left(\frac{-(k+1) e^{-k/u}}{u}\right)-k/u & \textup{for $u\leq k-2\sqrt{k}$}\\
-W_{-1}\left(\frac{-(k+1) e^{-k/u}}{u}\right)-k/u & \textup{for $u> k+2\sqrt{k}$}
\end{cases}
\end{equation}
For this choice of $\xi(u)$, we have that if $|u-k| \gg k^{1/2+\epsilon}$, then $\xi(u) \gg k^{-1/2+\epsilon}$.  Moreover, $\xi(u)<0$ for $u<k-2\sqrt{k}$, and $\xi(u) > 0$ for $u>k+2\sqrt{k}$.

\end{lemma5}
\begin{proof}
We have that 
\begin{eqnarray*}
\frac{d}{ds} \left(se^{(k+1)I(-s)}e^{us}\right)=e^{(k+1)I(-s)}e^{us}\left(1+s(u-(k+1)I'(-s))\right),
\end{eqnarray*}and this is $0$ when $s = -\xi(u)$ where $\xi(u)$ satisfies
$$(k+1)e^{\xi(u)} = k+u\xi(u).
$$In other words
\begin{equation}\label{xieqn}
\xi(u) = -W\left(\frac{-(k+1) e^{-k/u}}{u}\right)-k/u,
\end{equation}where $W$ is the Lambert W function.  Note that $\frac{-(k+1) e^{-k/u}}{u} \geq -1/e \Leftrightarrow (k+1)\leq ue^{\frac{k-u}{u}}$.  We first verify that the latter holds for all $|u-k|\geq 2\sqrt{k}$.  Indeed, a little calculus tells us that the function $ue^{\frac{k-u}{u}}$ has a global minimum on $[0, \infty)$ at $u = k$.  Since it is decreasing on $[0, k)$ and increasing on $[k, \infty)$, it suffices to check the assertion for $|u-k|= 2\sqrt{k}$.  But for $|u-k|= 2\sqrt{k}$, we have
\begin{eqnarray*}
ue^{\frac{k-u}{u}}
&=& k+\frac{(k-u)^2}{2u}+\frac{(k-u)^3}{3!u^2}+...\\
&\geq& k+\left(\frac{1}{2} - \frac{1}{3!}\right)\frac{(k-u)^2}{u}\\
&\geq& k+\frac{4}{3}\\
&>& k+1.
\end{eqnarray*}

Now let $u = k+E$, where $|E|>2\sqrt{k}$.  We examine two cases.  First, when $E<0$, we know that $-W_0(x) \leq 1$ for all $x\leq 0$ so
$$\xi(u) \leq 1 - \frac{k}{k+E} = \frac{E}{k+E}<0.$$
Next, when $E>0$, we know that $-W_{-1}(x) \geq 1$ for all $x\leq 0$ so
$$\xi(u) \geq 1 - \frac{k}{k+E} = \frac{E}{k+E}>0
$$Note that  $\left|\frac{E}{k+E}\right|\gg \frac{1}{k^{1/2-\epsilon}}$,
if $|E|\gg k^{1/2+\epsilon}$, and that $\xi(u)$ shares the same sign with $E$. 

\end{proof}

\begin{rmk1}
To motivate the definition of $\xi(u)$ in this lemma, note that $k/u$ is close to satisfying the equation defining $W\left(\frac{-(k+1)e^{-k/u}}{u}\right)$ so $k/u$ must sometimes be close to one of the branches.  The idea here is to take the other branch.  The sign change for $\xi(u)$ occurs near $u=k$, and this is also when the branches converge to the same value at $-\frac{1}{e}$.

\end{rmk1} 
We now need to estimate $\sigma(u)$ by Laplace inversion of $\hat{\sigma}(s)$ on the $\Re s = \Re \xi$ line.  For this purpose, we collect the following estimates.  

\begin{lemma5}\label{boundsforlatranslem}
Let $\xi$ be as in Lemma \ref{extremalem}.  Write $s = -\xi + i\tau$, with $\tau$ real, and assume $1<u\leq 10k$ with $|k-u|\gg k^{1/2+\epsilon}$.  Then for $|\tau|\geq k+u|\xi|$,
\begin{equation}\label{lem5eqn1}
\hat{\sigma}(s) = \frac{(k-1)!}{s^{k-1}}\left(1+O\left(\frac{u\xi + k}{|s|}\right)\right).
\end{equation}
Moreover, there exists $c >0$ such that for $|\tau| \leq \pi$, 
\begin{equation}\label{lem5eqn2}
\hat{\sigma}(s) \ll (k-1)! s e^{(k+1)(\gamma + I(\xi) )} e^{-c\frac{(k+1)}{|\xi|+1}\tau^2},
\end{equation}
and for $|\tau| > \pi$,
\begin{equation}\label{lem5eqn3}
\hat{\sigma}(s) \ll (k-1)! s e^{(k+1)(\gamma + I(\xi))}e^{-c\frac{(k+1)}{|\xi|+1}}.  
\end{equation}
\end{lemma5}
\begin{proof}
The first bound follows from $\hat{\sigma}(s) = \frac{(k-1)!}{s^{k-1}}e^{-(k+1)J(s)}$, and the bound $J(s) \ll |\frac{e^{\xi}}{s}| = |\frac{u\xi + k}{(k+1)s}|$.  
For the other two cases, set $H(\tau) = I(\xi) - I(-s) = \int_0^1 \frac{e^{h\xi}}{h}(1-e^{-i\tau h})dh$.  We extract real part to get that
\begin{eqnarray*}
\Re H(\tau)
&=& \int_0^{1} \frac{e^{h\xi}}{h}(1-\cos\tau h) dh 
\end{eqnarray*} 

For (\ref{lem5eqn2}), note that $1-\cos h\tau \geq \frac{2\tau^2h^2}{\pi^2}$ for $|\tau| \leq \pi$. We have that by the calculation in \cite{Ten} in Lemma 8.2
\begin{eqnarray*}
\Re H(\tau)
&\geq&\frac{\tau^2}{2\pi^2} \left|\int_0^{1} e^{h\xi} dh\right| \\
&\gg& \frac{\tau^2}{|\xi|+1}.
\end{eqnarray*}  From this and Lemma \ref{Latranslem}, we have (\ref{lem5eqn2}).

To prove the third bound (\ref{lem5eqn3}), observe that
\begin{eqnarray*}
\Re H(\tau)
&=& \int_0^{1} \frac{e^{h\xi}}{h}(1-\cos\tau h) dh \\
&\gg& \frac{1}{|\xi|+1}.
\end{eqnarray*}The last line follows from considering an open set $E \subset [0, 1]$ of small measure outside of which $(1-\cos\tau h) \gg 1$.  One may make $E$ small enough so that $\int_E e^{h\xi}dh$ is bounded by $\int_{[0, 1]\backslash E} e^{h\xi}dh$.  This is possible since $\xi \leq C$ for some absolute constant $C$ for $u$ in the specified range.  This is true in the case $u < k$ because $-W_0(x) \leq 1$ for $x\leq 0$ and it is true for $u > k$ since the argument inside $W_{-1}$ is bounded away from $0$ when $u\leq 10k$.
\end{proof}
We now apply the bounds above to obtain an estimate for $\sigma(t)$.  Set $\delta = \sqrt{\frac{\log^3 (k+1)}{c(k+1)}}$ where $c$ is the constant appearing in Lemma \ref{boundsforlatranslem}.  Let $K(u) = \frac{1}{2\pi}\int_{-\delta}^{\delta} \hat{\sigma}(s) e^{us} d\tau$, and $H(u) = \frac{1}{2\pi} \int_{\mathbb{R}\backslash[-\delta, \delta]} \hat{\sigma}(s) e^{us} d\tau$.  As above, we have written $s = -\xi + i\tau$.  We know that $\sigma(u) = K(u) + H(u)$, and we first find an upper bound for $H(u)$.

\begin{lemma5}\label{lemevalH}
Assume $k\geq 3$ and $u\gg \frac{k}{\log k}$ with $|k-u|\gg k^{1/2+\epsilon}$.  Then 
$$H(u) \ll (k-1)!e^{(k-1)(\gamma + I(\xi))} \frac{1}{(k+1)^{\log^2k}}.$$
\end{lemma5}
\begin{proof}
First note by (\ref{xieqn}) that $\xi\ll \log k$ when $u\gg \frac{k}{\log k}$.  Now, we split the integral in the definition of $H(u)$ into 3 ranges.  First, when $\delta < |\tau|\leq \pi$, we have by (\ref{lem5eqn2}) that the integral is 
\begin{eqnarray*}
&\ll& (k-1)!e^{(k-1)(\gamma + I(\xi))} \int_\delta^\infty e^{-c(k+1)\tau^2/\log k}d\tau\\
&\ll&(k-1)!e^{(k-1)(\gamma + I(\xi))} \frac{1}{\sqrt{k+1}^{1-\epsilon}} \int_{\log^{3/2} (k+1)}^\infty e^{-\tau^2}d\tau\\
&\ll&(k-1)!e^{(k-1)(\gamma + I(\xi))} \frac{1}{(k+1)^{\log^2k}}.
\end{eqnarray*}
Next, when $\pi < |\tau| \leq k+u|\xi|$, we get by (\ref{lem5eqn3}) that the integral is
$$\ll (k-1)! e^{(k-1)(\gamma + I(\xi))}e^{-k^{1-\epsilon}},
$$where we have used that $u\ll k$.  Lastly, for $|\tau| \geq k+u|\xi|$, we get by (\ref{lem5eqn1}) that the integral is
$$\ll (k-1)! \frac{1}{k^{k-1}},
$$which is tiny.
\end{proof}
Now we are ready to evaluate $K(u)$.
\begin{lemma5} \label{lemevalK}
Suppose that $k\geq 3$ and $\frac{k}{\log k} \ll u \leq 10k$ with $|k-u|\gg k^{1/2+\epsilon}$.  Then
$$K(u) = \frac{-(k-1)! \xi e^{(k+1)(\gamma+I(\xi))-u\xi}}{\sqrt{2\pi(k+1)I''(\xi)}}\left(1+ O\left(\frac{1}{(k+1)^{\epsilon}}\right)\right)
$$
\end{lemma5}
\begin{proof}
We first examine the Taylor expansion of $I(-s)$ about $\xi$.  First note that
$$I'(\xi) = \frac{e^{\xi}-1}{\xi} = \frac{u}{k+1} - \frac{1}{(k+1)\xi},
$$as before.  Thus
\begin{eqnarray*}
I(-s) &=& I(\xi) - \frac{i\tau u}{k+1}+\frac{i\tau}{(k+1)\xi} - \frac{\tau^2I''(\xi)}{2} + O(\tau^3).
\end{eqnarray*}

Since $\frac{1}{k^{1/2-\epsilon}} \ll \xi \ll \log k$ for $\frac{k}{\log k} \ll u\leq 10k$, we have that for $|\tau|\leq \delta$,
$$e^{(k+1)I(-s) + us} = e^{(k+1)I(\xi) - u\xi - (k+1)\frac{\tau^2I''(\xi)}{2} }\left(1+O\left(\frac{1}{k^{\epsilon}}\right)\right),
$$and so
\begin{eqnarray*}
K(u) &=& (k-1)! e^{(k+1)(\gamma+I(\xi))-u\xi}
\int_{-\delta}^\delta e^{ - (k+1)\frac{\tau^2I''(\xi)}{2}}(-\xi + i\tau)d\tau \times \\
&&\left(1+O\left(\frac{1}{k^{\epsilon}}\right)\right)\\
&=&-(k-1)! \xi e^{(k+1)(\gamma+I(\xi))-u\xi}
\int_{-\delta}^\delta e^{ - (k+1)\frac{\tau^2I''(\xi)}{2}}d\tau\left(1+O\left(\frac{1}{k^{\epsilon}}\right)\right),
\end{eqnarray*} by symmetry.  Note that
$$I''(\xi) = \frac{\xi e^{\xi} - e^{\xi} + 1}{\xi^2}.
$$Then for $u$ in the range specified, $\frac{1}{\log^2k} \ll I''(\xi)\ll 1$.
We also have that
\begin{eqnarray*}
\int_{-\delta}^\delta e^{ - (k+1)\frac{\tau^2I''(\xi)}{2}}d\tau
&=&\int_{-\infty}^\infty e^{ - (k+1)\frac{\tau^2I''(\xi)}{2}}d\tau + O\left(\frac{1}{\sqrt{I''(\xi)}(k+1)^{3/2}}\right)\\
&=& \sqrt{\frac{2\pi}{(k+1)I''(\xi)}}\left(1+ O\left(\frac{1}{(k+1)^{1/2}}\right)\right),
\end{eqnarray*}as desired.
\end{proof}

\begin{prop}
Say that $k\geq 3$ and $\frac{k}{\log k} \ll u \leq 10k$ with $|u-k| \gg k^{1/2+\epsilon}$.  Then
$$\sigma(u) = \frac{-(k-1)! \xi e^{(k+1)(\gamma+I(\xi))-u\xi}}{\sqrt{2\pi(k+1)I''(\xi)}}\left(1+ O\left(\frac{1}{(k+1)^{\epsilon}}\right)\right)
$$Moreover, by Lemma \ref{extremalem}, the first zero of $\sigma(u)$ must be $k+O(k^{1/2+\epsilon})$.
\end{prop}
\begin{proof}
The expression for $\sigma(u) = K(u)+H(u)$ follows directly from \ref{lemevalH} and \ref{lemevalK}.  Note that $I''(\xi) \gg \frac{1}{\log^2k}$ for $u$ in the range specified.  The last assertion follows from noting that $\sigma(u)$ changes sign when $\xi$ changes sign, and the fact that by Lemma \ref{lembasizero}, the first zero of $\sigma(u)$ must be $\gg \frac{k}{\log k}$.
\end{proof}
Finally, we note that Theorem $\ref{thmnegativresult}$ follows immediately from the above proposition.

\section{Cubic and Biquadratic Fields} \label{seccubicbi}
We now investigate the question of bounding the least non-split prime when $K$ is either cubic or biquadratic.  The general philosophy is the same for the two cases, although the technical details are different.  There is always a "trivial" bound which arises from considering cancellation in a quadratic character, and our purpose is to show that this bound can be improved.  In both cases, we benefit from interaction between a primary multiplicative function of interest and quadratic characters.  Simply put, if all the primes split up to the trivial bound, then the quadratic character is extremal and we may predict its behaviour far beyond the cancellation point.  In this case, the interaction with the primary multiplicative function produces a contradiction.  In order to obtain an actual bound, we need to understand what it means for a quadratic character to be close to extremal.

\subsection{Extremal behavior} \label{sectionextreme}
Let $\chi$ denote a quadratic character with modulus $q$ such that $\chi(p) = 1$ for all $p \leq y$ whenever $p\nmid q$.   We set $P(u) = \frac{1}{ \nu (y^u)} \sum_{p\leq y^u}\chi(p)\log p$, where $\nu(x) = \sum_{p\leq x}\log p$.  Also, let $\sigma(u) = \frac{1}{y^u}\sum_{n\leq y^u} \chi(n)$.  We further define 

$$I_{j}(u) = \int_{\substack{t_1+...+t_j \leq u\\t_i\geq 1\forall 1\leq i\leq j}}\prod_{i=1}^j\frac{1-P(t_i)}{t_i}dt_1...dt_j.$$
We remind the reader that $\sigma(u) = \sum_{j\geq 0}\frac{(-1)^jI_j(u)}{j!}$ where $I_0 \equiv 1$.  Note that the sum on the right hand side is finite.  Moreover, we have that $\sum_{j=0}^{2m-1}\frac{(-1)^jI_j(u)}{j!}\leq \sigma(u)\leq \sum_{j=0}^{2m}\frac{(-1)^jI_j(u)}{j!}$ for any $m\geq 0$.  Once again, we refer the reader to \cite{GS} for more details.

Let $A>0$ be such that $y^A = q^{1/4}$, so that $\sigma(u) = o(1)$ for $u>A$.  The reader should think of $A$ as being somewhat larger than $\sqrt{e}$.  The simple case when $A = \sqrt{e}$ is the extremal case appearing in the bound (\ref{burgesseqn}) and the behaviour of $P(t)$ here has been studied by other authors.  In their study of Beurling primes, Diamond, Montgomery, and Vorhauer reproduce the unpublished analysis of Heath-Brown on this subject in the appendix of \cite{DMV}.  This was also examined by Granville and Soundararajan \cite{GSBurgess} in an unpublished manuscript.  The lemma below quantifies the behaviour of $P(t)$ by comparing $\chi$ to an extremal character.  

\begin{lemma5} \label{lemquadchar}
Suppose that $\sqrt{e} \leq A\leq 2$, and set\footnote{$E$ measures the deviation of $A$ from $\sqrt{e}$.  In particular, $E = 0$ when $A = \sqrt{e}$.} $E = 2\log A - 1$.  Then the following holds.

1.  Say that we have some interval $(a, b) \subset (1, A)$.  Then 
\begin{equation*}
\int_a^b\frac{1-P(t)}{t}dt \geq 2\log \frac{b}{a} -E+o(1).
\end{equation*}
2.  For all $t \in [2, 3]$ but for a set of measure $0$, we have that
$$\frac{1 - P(t)}{t} = \frac{1}{2} \int_1^{t-1} \frac{1-P(u)}{u}\frac{1-P(t-u)}{t-u}du.
$$
Moreover, for all $t \in [2, 4]$ but for a set of measure $0$, we have that
$$\frac{1 - P(t)}{t} \leq \frac{1}{2} \int_1^{t-1} \frac{1-P(u)}{u}\frac{1-P(t-u)}{t-u}du.
$$
3. For all $t \in [2, 1+A]$ but for a set of measure $0$, we have that
$$\frac{4}{t}\log (t-1) - 2E \leq \frac{1 - P(t)}{t} \leq \frac{4}{t}\log (t-1)
$$
4.  For all $t \in [1+A, 3]$ but for a set of measure $0$, we have that
$$\frac{1 - P(t)}{t} \geq \frac{4}{t}\log \frac{A}{t-A} - 2E +o(1),$$ and for $t \in [3, 4]$, we have that
$$\frac{1-P(t)}{t} \geq \frac{4}{t}\log \frac{A}{t-A} - 2E -\frac{2}{3}(t-3)^2+o(1).
$$
Moreover, for  all $t \in [1+A, 2A]$ but for a set of measure $0$, we have that 
$$\frac{1 - P(t)}{t} \leq \frac{4}{t}\log \frac{A}{t-A}+o(1).
$$
\end{lemma5}
\begin{proof}
Note that $\sigma(u) = o(1)$ for $u>A$.  Thus $\int_1^A\frac{1-P(t)}{t}dt = 1+o(1)$ and the first assertion follows since $1-P(t) \leq 2$.

The second assertion follows from the fact that $P(t)$ is continuous almost everywhere, and when $P(t)$ is continuous, 
$$\frac{1-P(t)}{t}  = \lim_{\epsilon \rightarrow 0} \frac{1}{\epsilon} \left(I_1(t+\epsilon) - I_1(t)\right).$$  For $t \in [2, 3]$, 
$$I_1(t+\epsilon) - I_1(t) = \frac{1}{2}(I_2(t+\epsilon) - I_2(t)),$$ and for $t \in [2, 4]$,  
\begin{eqnarray*}
I_1(t+\epsilon) - I_1(t) 
&=& \frac{1}{2}\left(I_2(t+\epsilon) - I_2(t)\right) - \frac{1}{6}\left(I_3(t+\epsilon) - I_3(t)\right) \\
&\leq& \frac{1}{2}\left(I_2(t+\epsilon) - I_2(t)\right).
\end{eqnarray*}
Thus, it remains to evaluate 
\begin{eqnarray*}
\lim_{\epsilon \rightarrow 0}\frac{1}{2\epsilon}\left(I_2(t+\epsilon) - I_2(t)\right)
&=&  \lim_{\epsilon \rightarrow 0}\frac{1}{2\epsilon}\int_{\substack{t \leq t_1+t_2 \leq t+\epsilon\\ t_1, t_2\geq 1}}\frac{1-P(t_1)}{t_1}\frac{1-P(t_2)}{t_2}dt_1dt_2\\
&=& \frac{1}{2}\int_1^{t-1}\frac{1-P(t_1)}{t_1}\frac{1-P(t-t_1) }{t - t_1}dt_1,
\end{eqnarray*}almost everywhere, as desired.

To prove the upper bound in the third assertion, note that 
$$\frac{1}{2} \int_1^{t-1} \frac{1-P(u)}{u}\frac{1-P(t-u)}{t-u}du \leq 
2 \int_1^{t-1} \frac{1}{u}\frac{1}{t-u}du
= \frac{4}{t}\log (t-1).
$$To prove the lower bound in the third assertion, we let $f(t) = \frac{1-P(t)}{t}$ and $m(t) = \frac{2}{t} \geq f(t)$ for all $t$.  Then we have that for $t \in [2, 1+A]$,
\begin{eqnarray*}
\int_1^{t-1} f(u)f(t-u)du
&=& \int_1^{t-1} (f(u)-m(u))f(t-u)du \\
&+& \int_1^{t-1} m(u)(f(t-u)-m(t-u))du 
+ \int_1^{t-1} m(u)m(t-u)du\\
&\geq& \frac{8}{t} \log (t-1) -4E.
\end{eqnarray*}Here we have bounded the first two terms from below both by $-2E$ using the first assertion and that $f(u)\leq m(u) \leq 2$ for all $u\in [1, A]$.

The proof of the fourth assertion is similar.  The only difference in the proof of the first and last bounds arises from the fact that $\int_A^2\frac{1-P(u)}{u}du = o(1)$.  Thus for $1+A \leq t = 1+A+\delta \leq 2A$,
$$\int_1^{t-1} \frac{1-P(u)}{u}\frac{1-P(t-u)}{t-u}du = \int_{1+\delta}^{t-1-\delta} \frac{1-P(u)}{u}\frac{1-P(t-u)}{t-u}du + o(1).
$$For the second bound in the fourth assertion, one also needs to use that
\begin{eqnarray*}
\frac{1-P(t)}{t}
&=& \lim_{\epsilon \rightarrow 0} \frac{1}{\epsilon}\left(\frac{1}{2}\left(I_2(t+\epsilon) - I_2(t)\right) - \frac{1}{6}\left(I_3(t+\epsilon) - I_3(t)\right)\right) \\
&\geq& \frac{1}{2}\lim_{\epsilon \rightarrow 0} \frac{1}{\epsilon}\left(I_2(t+\epsilon) - I_2(t)\right) - \frac{2}{3}(t-3)^3.
\end{eqnarray*}This follows from the calculation that
\begin{eqnarray*}
\lim_{\epsilon \rightarrow 0} \frac{1}{\epsilon}\left(I_3(t+\epsilon) - I_3(t)\right) 
&=& \int_{\substack{t_1+t_2\leq t-1\\t_1, t_2\geq 1}}\frac{1-P(t_1)}{t_1}\frac{1-P(t_2)}{t_2}\frac{1-P(t-t_1-t_2)}{t-t_1-t_2}dt_1dt_2\\
&\leq& 8\int_{\substack{t_1+t_2\leq t-1\\t_1, t_2\geq 1}}\frac{1}{t_1}\frac{1}{t_2}\frac{1}{t-t_1-t_2}dt_1dt_2\\
&\leq& 8\frac{(t-3)^2}{2},
\end{eqnarray*}upon calculating the volume of the region of integration.
\end{proof}

\subsection{Cubic Fields}
Let $K$ be a cubic field.  In this case, it is easy to see that a much better result than $\N\ll d_K^{\frac{1}{2(l-1)}}$ is possible.  In the case where $K$ is Galois, then $K$ must necessarily be abelian, and so $\zeta_K(s) = \zeta(s)L(s, \chi_1)L(s, \chi_2)$ for some Dirichlet characters $\chi_1$ and $\chi_2$ with conductors $q_1$ and $q_2$ respectively.  Say that $q_1 \leq q_2$.  Then since $\chi_1$ has order $3$, by Lemma 2.4 of \cite{DRHB}, $\chi_1(n)$ exhibits cancellation by $q^{1/4+\epsilon}$.  Thus $\N \ll_\epsilon q_1^{1/4+\epsilon} \ll d_K^{1/8+\epsilon}$.  Clearly, a stronger statement should be possible in the abelian case, but we shall be more interested in the general case here.

For the rest of this section, say that $K$ is not Galois.  Then
$$\zeta_K(s) = \zeta(s) L(f, s),
$$where $f$ is a holomorphic modular Hecke eigenform of weight $k$ and level $N$.  We also have that the $L$-function associated to $f$ is of the form
$$L(f, s) = \prod_p \left(1-\frac{\alpha_p}{p^s}\right)^{-1}\left(1-\frac{\beta_p}{p^s}\right)^{-1}
= \prod_p \left(1-\frac{a(p)}{p^s} + \frac{\chi(p)}{p^{2s}}\right)^{-1},
$$where $\chi$ is a quadratic character with modulus $q \leq d_K$.  Visibly from the Euler product above, we have that $p$ cannot split in $K$ if $\chi(p) = -1$.  Thus, 
\begin{equation}\label{eqncubicbound}
\N \ll d_K^{1/4\sqrt{e}+o(1)}.
\end{equation}  This is the starting point for our investigation.

Let $f(n)$ be the completely multiplicative function with $f(p) = a(p)$ for all primes $p$.  Then $f(n)$ exhibits cancellation by $d_K^{1/2+o(1)}$.  We now try to improve the bound of $\N \ll d_K^{1/4\sqrt{e}}$ by leveraging information about the two multiplicative functions $f(n)$ and $\chi(n)$.

\begin{rmk2}
Our main focus here is to show that improvments over the bound (\ref{eqncubicbound}) are possible.  For simplicity, we will not attempt to completely optimize our calculations.  In particular, we do not use the available subconvexity result for $\zeta_K(s)$ which show that $f(n)$ exhibits cancellation by $d_K^{1/2-\delta}$ for some $\delta >0$ (see Appendix A of \cite{EEMV} for a synopsis of known results).  
\end{rmk2}

As in \S \ref{sectionextreme}, let $P(t)$ denote the average over primes of $f(p)$ and let $P'(t)$ denote the same average for $\chi(p)$.  Let \footnote{This definition of $\sigma(t)$ differs from the definition in \S 3 by a factor of $t$.}$\sigma(t) = \frac{1}{y^t\log y^t}\sum_{n\leq y^t}f(n)$.  Also, as in \S \ref{sectionextreme}, assume that there exists some $y = d_K^A$ such that all primes $p\leq y$ split completely, where we may assume that $A>\frac{1}{8}$.

We begin by quantifying the relationship between $f(p)$ and $\chi(p)$.

\begin{lemma5} \label{lemcubictech}
With $f$ and $\chi$ as above, we have that $f(p) \geq -\frac{\chi(p)+1}{2}$ for all unramified primes $p$.  It follows that $P(t) \geq -\frac{P'(t)+1}{2} + o(1),$where the $o(1)$ is a quantity tending to $0$ as $d_K \rightarrow \infty$ uniformly for $t\geq 1$.
\end{lemma5}
\begin{proof}
This follows from the fact that $f(p) = \alpha_p + \beta_p$ and $\chi(p) = \alpha_p \beta_p$.  First assume that $p$ is unramified.  There are three possibilities to check corresponding to the three possibilities for the local factor at $p$ in $\zeta_K(s)$ which is always of the form $\prod_{\mathfrak{p}|p}\left(1-\frac{1}{N(\mathfrak{p})^s}\right)^{-1}$.  When $p$ splits completely, the local factor is
$$\left(1-\frac{1}{p^s}\right)^{-3},
$$
and so $\alpha_p = \beta_p = 1$ whence $f(p) = 2$ and $\chi(p) = 1$.  When $p$ is inert, 
the local factor is of the form
$$\left(1-\frac{1}{p^{3s}}\right)^{-1},
$$so $\alpha_p = 1/\beta_p = e^{\pm 2\pi i/3}$ and $f(p) = -1$ and $\chi(p) = 1$.  In the remaining case, $p$ factors as $p = \mathfrak{p}_1\mathfrak{p}_2$ where the norms of the ideals on the right are $p$ and $p^2$, and so the local factor is of the form
$$\left(1-\frac{1}{p^s}\right)\left(1-\frac{1}{p^{2s}}\right).
$$
Thus, in this case, $\alpha_p = -\beta_p = \pm 1$ and $f(p) = 0$ and $\chi(p) = -1$.  In all three cases, we have verified that $f(p)  \geq -\frac{\chi(p)+1}{2}$.  The statement about the averages $P(t)$ and $P'(t)$ follows by definition, and since the number of ramified primes is bounded by $\log d_K$, and hence contribute at most $O\left(\frac{\log^2 d_K}{y}\right) = O\left(\frac{\log^2 d_K}{d_K}\right) = o(1)$.
\end{proof}

\subsubsection{Outline of proof:}  Our bound for $\N$ will result from a lower bound for the first zero of $\sigma(t)$, which we know must eventually be identically zero by cancellation.  The Lemma above combined with the Proposition \ref{propstructure} tells us that we can instead study the first zero of the solution to (\ref{eqnconvolution}) with $-\frac{P'(t)+1}{2}$ in place of $P(t)$.  We then use our estimates for $P'(t)$ from Lemma \ref{extremalem} to finish the proof.

$\\$
We let $$I_j(u) = \int_{\substack{t_1+...+t_j \leq u\\t_k\geq 1\forall 1\leq k\leq j}}\left(\frac{u-\sum_{k=1}^jt_k}{u}\right)\prod_{k=1}^j\frac{(2-P(t_k))}{t_k}dt_1...dt_j.$$  Then for $u\leq 4$,
$$\sigma(u) = 1-I_1(u) + \frac{I_2(u)}{2} - \frac{I_3(u)}{6}.
$$
Set 
$$I_3'(u) = \int_{\substack{t_1+t_2+t_3 \leq u\\t_k\geq 1\forall 1\leq k\leq 3}}\frac{u-t_1-t_2-t_3}{ut_1t_2t_3}dt_1dt_2dt_3.$$ 
Note that
$$\frac{I_3(u)}{6} \leq \frac{9}{2}I_3'(u),$$
where we have used the trivial bound $2-P(t) \leq 3$.
We thus have that
\begin{equation} \label{eqncubiclowerbound}
\sigma(u) \geq 1-I_1(u) + \frac{{I_2}(u)}{2} - \frac{9}{2}I_3'(u).
\end{equation}
By Proposition \ref{propstructure} and Lemma \ref{lemcubictech}, we know that (\ref{eqncubiclowerbound}) still holds when $P(t)$ is replaced by $-\frac{P'(t)+1}{2}$.  Henceforth, assume that $P(t) = -\frac{P'(t)+1}{2}$ for all $t\geq 1$.  Now, we calculate an upper bound for $I_1(u)$.  

\begin{lemma5}\label{lemcubiclowerI1}
For notational convenience, set $g(t, u) = g(t) = \frac{u-t}{tu}$.  For all $t\in [A, 4]$ but for a set of measure zero, we have $-P(t) \leq U(t)$ where
\begin{equation*}
U(t)= 
\begin{cases} 1 & \text{if $A<t\leq 2$,}
\\
\min(1, 1-2\log (t-1) + Et) &\text{if $2< t\leq 1+A$}
\\
\min(1, 1-2\log \left(\frac{A}{t-A}\right) + Et) &\text{if $1+A< t\leq 3$}
\\
\min(1, 1-2\log \left(\frac{A}{t-A}\right) + Et+\frac{1}{3}t(t-3)^3) &\text{if $3\le t\leq 4$.}
\end{cases}
\end{equation*}
Let $u=2A \leq 4$.  Then, 
\begin{eqnarray*}
\int_1^u (2-P(t))g(t)dt &\leq& 2\int_1^ug(t)dt + \int_A^u U(t)g(t)dt\\
&+& \frac{1}{2}\left(\log A -1+\frac{1}{u}(1+A- \frac{2A}{\sqrt{e}})+ \int_1^Ag(t)dt \right)
\end{eqnarray*}
\end{lemma5}
\begin{proof}
Since we assume that $P(t) = -\frac{P'(t)+1}{2}$ and Lemma \ref{lemquadchar} applies to $P'(t)$, $-P(t) \leq U(t)$ for $A\leq t\leq u$.  Hence,
$$\int_1^u (2-P(t))g(t)dt \leq 2\int_1^ug(t)dt + \int_A^u U(t)g(t)dt+\int_1^A\frac{1+P'(t)}{2}g(t)dt,
$$and moreover,
$$\int_1^A\frac{1+P'(t)}{2}g(t)dt = \frac{1}{2} \left(\int_1^A g(t)dt + \int_1^A\frac{P'(t)}{t} dt - \int_1^A \frac{P'(t)}{u}dt\right).
$$

We know that $\int_1^A \frac{1-P'(t)}{t}dt = 1$, so $\int_1^A \frac{P'(t)}{t}dt = \log A - 1$.  Thus, $\int_1^A P'(t)dt \geq \int_1^{A/\sqrt{e}}1dt-\int_{A/\sqrt{e}}^A 1dt = 2A/\sqrt{e} - 1-A$.  To see this, let 
$$
\gamma(t) = \begin{cases}
1 & \text{if $1\leq t\leq A/\sqrt{e}$,}
\\
-1 & \text{if $A/\sqrt{e}< t\leq A$.}
\end{cases}
$$Note that $\int_1^A \frac{\gamma(t)}{t}dt = \log A-1$.  Let $\lambda(t) : [1, A] \rightarrow [-1, 1]$ be any such function with $\int_1^A \frac{\lambda(t)}{t}dt = \log A-1$ and let $h(t) = \lambda(t) - \gamma(t)$.  It suffices to show that $\int_1^Ah(t) dt\geq 0.$  We have that 
$$A/\sqrt{e} \int_1^{A/\sqrt{e}} \frac{h(t)}{t}dt + A/\sqrt{e} \int_{A/\sqrt{e}}^A \frac{h(t)}{t}dt = 0.$$  Note that $h(t) \leq 0$ for $1\leq t\leq A/\sqrt{e}$ and $h(t) \geq 0$ for $A/\sqrt{e}<t\leq A$.  Thus we have $$A/\sqrt{e} \int_1^{A/\sqrt{e}} \frac{h(t)}{t}dt \leq \int_1^{A/\sqrt{e}}h(t)dt$$
and 
$$A/\sqrt{e} \int_{A/\sqrt{e}}^A \frac{h(t)}{t}dt \leq \int_{A/\sqrt{e}}^Ah(t)dt.$$  Adding the two immediately produces the desired result.

From this, we get that
$$\int_1^A\frac{1+P'(t)}{2}g(t)dt \leq \frac{1}{2}\left(\log A -1+\frac{1}{u}(1+A- 2A/\sqrt{e}) + \int_1^Ag(t)dt \right).
$$
\end{proof}

We now need a lower bound for $I_2(u)$.

\begin{lemma5}
Let
\begin{equation*}
L(t)= 
\begin{cases} 0 & \text{if $1\leq t\leq A$,}
\\
1 & \text{if $A<t\leq 2$,}
\\
\min(1, 1-2\log (t-1)) &\text{if $2< t\leq 1+A$}
\\
\min(1, 1-2\log \left(\frac{A}{t-A}\right)) &\text{if $1+A< t\leq 2A.$}
\end{cases}
\end{equation*}
Then for all $t\in[1, 2A]$ but a set of measure zero we have that $-P(t) \geq L(t)$.  Thus for $u=2A$,
$$I_2(u) \geq \int_{\substack{t_1+t_2 \leq u\\t_k\geq 1}}\frac{(2+L(t_1))}{t_1}\frac{(2+L(t_2))}{t_2}\frac{u-t_1-t_2}{u}dt_1dt_2
$$
\end{lemma5}
\begin{proof}
The proof is immediate from Lemma \ref{lemquadchar}, and the fact that we have set $P(t) = -\frac{1+P'(t)}{2}$.
\end{proof}

We now proceed to prove the Theorem.
\begin{proof}
Preserve the notation from the Lemma above.  Since $\sigma(u) = o(1)$ for $u=2A$, we have that
\begin{eqnarray*}
o(1)&\geq& 1- 2\int_1^ug(t)dt + \int_A^u U(t)g(t)dt+\int_1^A\frac{1+P'(t)}{2}g(t)dt \\
&+& \frac{1}{2}\int_{\substack{t_1+t_2 \leq u\\t_k\geq 1}}\frac{(2+L(t_1))}{t_1}\frac{(2+L(t_2))}{t_2}\frac{u-t_1-t_2}{u}dt_1dt_2 - \frac{9}{2}I_3'(u).
\end{eqnarray*}

Using Maple and the above lemmas, we can check that the right side of the above inequality is positive when $A=1.6625$.  Thus for the inequality above to be true, $A>1.6625$ so $4A > 6.65$, and since $\N \ll_\epsilon d_K^{\frac{1}{4A}+\epsilon}$, we must have that 
$$\N \ll d_K^{\frac{1}{6.65}}.
$$
The number $6.65$ should be compared with $4\sqrt{e} = 6.59...$
\end{proof}

\subsection{Biquadratic Fields}
We now fix $K$ to be a biquadratic field.  Then $\zeta_K(s) = \zeta(s)L(s, \chi_1)L(s, \chi_2)L(s, \chi_1\chi_2)$, where $\chi_1$ and $\chi_2$ are quadratic characters with modulus $q_1$ and $q_2$ say.  Finding the smallest non-split prime is the same as finding the smallest prime which is a quadratic non-residue for either $q_1$ or $q_2$.  Clearly, the trivial bound here is of the form $\N \ll_\epsilon \min(q_1, q_2)^{\frac{1}{4\sqrt{e}}+\epsilon}$ arising immediately from the discussion in the introduction.  Our purpose here is to show that more information can be gleaned from considering the behaviour of $\chi := \chi_1\chi_2$ in conjunction with that of $\chi_1$ and $\chi_2$.  Let $q = \max(q_1, q_2)$; we will only use the fact that both $\chi_i$ exihibt cancellation by $q^{\frac{1}{4}+o(1)}$.  Note that if $q_1$ and $q_2$ are far apart, then we expect to derive little information from the interaction of $\chi_1$ and $\chi_2$.  This will be reflected in the discussion at the end of this section.

Assume that all the primes split up to $y$.  Here, the reader may find it helpful to think of $y$ as being a slightly smaller power of $q_1q_2$ than the trivial bound.  We set $P_i(u) = \frac{1}{ \nu (y^u)} \sum_{p\leq y^u}\chi_i(p)\log p$, for $i=1, 2$ and where $\nu(x) = \sum_{p\leq x}\log p$.  Similarly, we set $P(u)= \frac{1}{ \nu (y^u)} \sum_{p\leq y^u}\chi(p)\log p$.  Finally, define $\sigma_i(u)$ for $i\in\{1, 2\}$, and $\sigma(u)$ as in \S \ref{sectionextreme}.

We also define $$I_{i, j}(u) = \int_{\substack{t_1+...+t_j \leq u\\t_k\geq 1\forall 1\leq k\leq j}}\prod_{k=1}^j\frac{1-P_i(t_k)}{t_k}dt_1...dt_j,$$ and similarly 
$$I_j(u) = \int_{\substack{t_1+...+t_j \leq u\\t_k\geq 1\forall 1\leq k\leq j}}\prod_{k=1}^j\frac{1-P(t_k)}{t_k}dt_1...dt_j.$$
We begin with the following basic observation.
\begin{lemma5} \label{lembiquadPu}
Let 
$$S_1 = \frac{1}{\nu(y^u)}\sum_{\substack{p\leq y^{u} \\ \chi_1(p) = \chi_2(p) = 1}}\log p,$$ and 
$$S_{-1} = \frac{1}{\nu(y^u)}\sum_{\substack{p\leq y^{u} \\ \chi_1(p) = \chi_2(p) = -1}}\log p.$$  Then,
$$P(u)  = 2S_1 + 2S_{-1} -1 + o(1).$$Furthermore, if $P_i(t) \geq \alpha > 0$ for all $i\in \{1, 2\}$, or if $P_i(t) \leq -\alpha <0$ for all $i\in \{1, 2\}$, then
$$P(u) \geq 2\alpha -1.
$$

\end{lemma5}
\begin{proof}
Let 
$$S_{1, -1}(u) = \frac{1}{\nu(y^u)}\sum_{\substack{p\leq y^{u} \\ \chi_1(p) = -\chi_2(p) = 1}}\log p,
$$and similarly define $S_{-1, 1}(u)$.  Then we have that
$$S_1+S_{-1}+S_{1, -1}+S_{-1, 1} = 1+o(1),
$$where the $o(1)$ comes from the ramified primes.
Since $\chi(p) = \chi_1(p)\chi_2(p)$, we also have that
$$P(u) = S_1(u) + S_{-1}(u) - S_{1, -1}(u) - S_{-1, 1}(u).
$$Adding the two equations give the first portion of the Lemma.
Now say that $P_i(t) \geq \alpha > 0$ for all $i\in \{1, 2\}$.  Then since $
\alpha \leq P_1(t) = S_1(t) -S_{-1}(t) + S_{1, -1}(t) - S_{-1, 1}(t)$ and $\alpha \leq P_2(t) = S_1(t) -S_{-1}(t) - S_{1, -1}(t) + S_{-1, 1}(t)$, we have that $2\alpha \leq 2(S_1(t) -S_{-1}(t)) \leq P(t) + 1$, as desired.  The remaining assertion is proven in the exact same way. 
\end{proof}
\subsubsection{Outline of proof:}  As in \S 4.2, our bound for $\N$ will result from a lower bound for the first zero of $\sigma(t)$, which we know must eventually be identically zero by cancellation.  The Lemma above relates the behaviour of $P(t)$ with expressions $P_1(t)$ and $P_2(t)$ which may be estimated by Lemma \ref{extremalem}.

\begin{lemma5}\label{lembiquadtech}
Let $A$ be such that $y^A = q^{1/4}$, and $B\leq 2A$ be such that $y^B = (q_1q_2)^{1/4}$.  Then,
$$0\geq 3 - 4\log A - \int_2^{B}\frac{1-P(t)}{t}dt + o(1).
$$
\end{lemma5}
\begin{proof}
We have that $0 = \sigma_i(u) = 1 - I_{i, 1}(u)$ for $A\leq u\leq 2$.  Adding this for $i=1, 2$, we get
\begin{eqnarray*}
\log u - 1 &=& \int_1^u \frac{P_1(t)+P_2(t)}{2t}dt\\
&=& \int_1^u \frac{S_1(t)-S_{-1}(t)}{t}dt.
\end{eqnarray*}  Rearranging, and noting that $S_1(t) \geq 0$, we get that $\int_1^u \frac{S_{-1}(t)}{t}\geq 1-\log u.$
Hence by the previous Lemma
$$\int_1^u \frac{P(u)}{u}du \geq \int_1^u \frac{2S_{-1}(t)-1}{t}dt \geq 2-3\log u.
$$Thus, rearranging again, and setting $u=  A$, we get that 
$$1 - \int_1^A \frac{1-P(u)}{u}du \geq 3-4\log A + o(1).$$  Observe that $\int_A^2 \frac{1-P_i(u)}{u}du = o(1)$ for each $i$ and so $\int_A^2 \frac{1-P(u)}{u}du = o(1)$ also.  
We thus have that
$$o(1) = \sigma(B) \geq 1 - I_1(B)\geq 3 - 4\log A - \int_2^{B}\frac{1-P(t)}{t}dt .
$$
\end{proof}
Lemma \ref{lembiquadPu} would give us a non trivial upper bound\footnote{By nontrivial, we mean that it must be smaller than the trivial bound given by $1-P(t)\leq 2$.} for $\int_2^{B}\frac{1-P(t)}{t}dt$ provided that we have sufficient information about $\chi_1$ and $\chi_2$.  The latter is furnished by Lemma \ref{lemquadchar}.  We collect the calculations and prove the theorem below.

\begin{proof}
For $2\leq u\leq 1+A$, we have by Lemma \ref{lemquadchar} and \ref{lembiquadPu} that $P(t) \geq 1-8\log (t-1)$.  Hence 
$$\int_2^{1+e^{1/4}}\frac{1-P(t)}{t}dt \leq \int_2^{1+e^{1/4}}\frac{8\log (t-1)}{t}dt < 0.13538.$$

In the range $2\leq u\leq 1+A$, we have by Lemma \ref{lemquadchar} that $P_i(t) \leq 1-4\log (t-1) +2Et$.  By Lemma \ref{lembiquadPu}, we have that $1-P(t) \leq 4(1-2\log (t-1) +Et)$.  This bound is only meaningful when the right hand side is $\leq 2$.  Thus, let $t_0<1+A$ be such that $2(1-2\log (t_0-1) +Et_0)=1$.  Then 
$$\int_{t_0}^{1+A} \frac{1-P(t)}{t}dt \leq 4\int_{t_0}^{1+A} \left(\frac{1-2\log (t-1)}{t} + E\right)dt.
$$
Further, in the range $1+A \leq u \leq 3$, we have by Lemma \ref{lemquadchar}  that $P_i(t) \leq 1-4\log \frac{A}{t-A} +2Et+ o(1)$.  By Lemma \ref{lembiquadPu}, we have that $\frac{1-P(t)}{t} \leq 4\frac{1-2\log \frac{A}{t-A} +Et}{t}+ o(1)$.  Let $t_1>1+A$ be such that $2(1-2\log \frac{A}{t-A} +Et_1)=1$.  Then 
$$\int_{1+A}^{t_1} \frac{1-P(t)}{t}dt \leq 4\int_{1+A}^{t_1} \left(\frac{1-2\log \frac{A}{t-A}}{t} + E\right)dt+ o(1).
$$
Let $t_2 = \frac{A(1+e^{1/4})}{e^{1/4}}$.  In the range, $t_2 \leq u \leq B \leq 2A$, we have by Lemma \ref{lemquadchar} that $1-P_i(t) \leq 4\log \frac{A}{t-A} + o(1)$.  Then similarly, we get that 
$$\int_{t_2}^{B} \frac{1-P(t)}{t}dt \leq 8\int_{t_2}^B \frac{\log \frac{A}{t-A}}{t}dt.
$$

We use the trivial bound of $1-P(t) \leq 2$ for the range not given above.  For any given $B$, the preceding discussion gives us an upper bound for $\int_2^{B}\frac{1-P(t)}{t}dt$ and we may derive a lower bound for $A$ by Lemma \ref{lembiquadtech} which states that  
$$4\log A \geq 3  - \int_2^{B}\frac{1-P(t)}{t}dt + o(1).$$

Without loss of generality, say that for some $\delta\geq 0$ that $q_1=q^{1-\delta}$ and $q_2=q$, and note that $B = (2-\delta)A$.  If $q_1$ is much smaller compared to $q_2$, then we expect to derive little benefit from the above and then our bound will be $\N \ll q^{\frac{1-\delta}{4\sqrt{e}}}$.    The rest is a numerical optimization using Maple over values of $\delta$ from which we derive that the worst value for $\delta$ occurs when $\delta = 0.061...$ and then
$$\N \ll q^{0.142}.
$$ or equivalently,
$$\N \ll (q_1q_2)^{\frac{0.146}{2}}.
$$
When $q_1 \asymp q_2 =q$, $\delta=0$ and we have that
$$\N \ll (q_1q_2)^{\frac{0.141}{2}}.
$$

\begin{rmk2}
The reader may be curious about whether this result might be improved if we included the $I_2(u)$ and $I_3(u)$ terms, as we did in the cubic case.  While we may improve the result with enough care, the possible improvements here are limited.  The reason is because when $1\leq t\leq A$, we expect $P_i(t)$ to be close to $-1$ and when $A< t\leq 2$, we have that $P_i(t) = 1$.  Thus $P(t)$ is close to $1$ for $1\leq t\leq 2$.  Hence for $u\leq 4$, it would be reasonable to expect $I_2(u)$ and $I_3(u)$ to be fairly small.
\end{rmk2}

\end{proof}
\paragraph{Acknowledgements:} I would like to express my gratitude to Professor Soundararajan for very generously sharing his time and ideas on various topics in this paper, as well as for his constant encouragement throughout.  I also wish to thank Vorrapan Chandee for a careful reading of this paper.  I am grateful to the referee for many helpful editorial remarks.

\end{document}